\documentclass[12pt, twoside]{article}
\usepackage{amsmath,amsthm,amssymb}
\usepackage{times}
\usepackage{enumerate}
\usepackage{verbatim}
\usepackage[toc,page]{appendix}
\usepackage{mathrsfs}

\newtheorem{thm}{Theorem}[section]

\newtheorem{lem}[thm]{Lemma}
\newtheorem{cor}[thm]{Corollary}

\theoremstyle{definition}

\theoremstyle{remark}

\def\titlerunning#1{\gdef\titrun{#1}}
\makeatletter
\def\author#1{\gdef\autrun{\def\and{\unskip, }#1}\gdef\@author{#1}}
\def\address#1{{\def\and{\\\hspace*{18pt}}\renewcommand{\thefootnote}{}%
\footnote {#1}}%
\markboth{\autrun}{\titrun}}
\makeatother
\def\email#1{e-mail: #1}
\def\subjclass#1{{\renewcommand{\thefootnote}{}%
\footnote{\emph{Mathematics Subject Classification (2010):} #1}}}
\def\keywords#1{\par\medskip
\noindent\textbf{Keywords.} #1}

\theoremstyle{definition}

\newtheorem{rem}[thm]{Remark}

\numberwithin{equation}{section}

\newcommand{\R}{\mathbf{R}}

\newcommand{\C}{\mathbf{C}}
\newcommand{\Z}{\mathbf{Z}}
\newcommand{\Q}{\mathbf{Q}}
\newcommand{\HH}{\mathbf{H}}
\newcommand{\PSL}{\mathbf{PSL}}
\newcommand{\kk}{\mathbf{k}}
\newcommand{\J}{\mathbf{J}}
\newcommand{\K}{\mathbf{K}}
\newcommand{\Mod}[1]{\ (\textup{mod}\ #1)}

\providecommand{\Li}{\operatorname{Li}}

\begin{document}


\baselineskip=17pt


\titlerunning{Prime Geodesic Theorem for the Picard manifold}

\title{Prime Geodesic Theorem for the Picard manifold}

\author{Olga Balkanova
\and 
Dmitry Frolenkov}

\date{}

\maketitle

\address{O. Balkanova: Department of Mathematical Sciences, Chalmers University of Technology and University of Gothenburg,  Chalmers tv\"{a}rgata 3,  Gothenburg 412 96,
   Sweden; \email{olgabalkanova@gmail.com}
\and
D. Frolenkov: Steklov Mathematical Institute of Russian Academy of
Sciences, 8 Gubkina st., Moscow, 119991, Russia; \email{frolenkov@mi.ras.ru}}


\begin{abstract}
Let $\Gamma=\PSL(2,\Z[i])$ be the Picard group and $\HH^3$ be the three-dimensional hyperbolic space.
We study the Prime Geodesic Theorem for the quotient $\Gamma  \setminus \HH^3$, called the Picard manifold, obtaining an error term of size $O(X^{3/2+\theta/2+\epsilon})$, where $\theta$ denotes a subconvexity exponent for quadratic Dirichlet $L$-functions defined over Gaussian integers.

\end{abstract}

\section{Introduction}
The classical Prime Geodesic Theorem states that the counting function $\pi(X)$ of primitive hyperbolic classes in $\PSL_2(\Z)$ whose norm does not exceed $X$ satisfies the asymptotic law
$$\pi(X)\sim \Li(X)\text{ as }X \rightarrow \infty,$$
where $\Li(X)$ is the logarithmic integral.
This theorem and its generalisations can be considered as geometric analogues of the Prime Number Theorem, while norms of primitive hyperbolic elements are sometimes called "pseudoprimes".

In this paper we study the three dimensional version of the Prime Geodesic Theorem. Different concepts of the two dimensional theory can be extended to this case in a natural and elegant way. The role of the Poincar\'{e} upper half plane $\HH^2$ is now played by the three dimensional hyperbolic space
\begin{equation}\label{3dspace}
\HH^3=\{(z,r); \quad z=x+iy \in \C; \quad r>0\}.
\end{equation}
Let $\Gamma \subset \PSL (2,\C)$ be a discrete cofinite group. Prime Geodesic Theorem for the hyperbolic manifold $\Gamma \setminus \HH^3$ provides an asymptotic formula for the function $\pi_{\Gamma}(X)$, which counts the number of primitive hyperbolic or loxodromic elements in $\Gamma$ with norm less than or equal to $X$.

In the pioneering paper \cite{Sarnak}, Sarnak proved an asymptotic formula for $\pi_{\Gamma}(X)$ with an error term of size 
\begin{equation}\label{error:Sarnak}
O(X^{3/2+1/6+\epsilon}).
\end{equation}

Further progress has been made in the case of the Picard group $\Gamma=\PSL(2,\Z[i])$ defined over Gaussian integers
$$\Z[i]=\{a+bi; \quad a,b \in \Z\}.$$

Assuming the "mean Lindel\"{o}f" hypothesis for symmetric square $L$-functions attached to Maass forms on the Picard manifold $\Gamma  \setminus \HH^3$, Koyama \cite{Koyama} obtained  an error term of size
\begin{equation}\label{est:koyama}
O(X^{3/2+1/14+\epsilon}).
\end{equation}

Finally, an unconditional improvement of Sarnak's estimate \eqref{error:Sarnak}, namely 
\begin{equation}\label{error:CCL}
O(X^{3/2+1/8+\epsilon}),
\end{equation} 
 was derived in  \cite{ChatCherLaak} as a consequence of a non-trivial estimate for the second moment of symmetric-square $L$-functions.

Proofs of the last two results are centered around Nakasuji's explicit formula proved in \cite[Theorem 4.1]{Nak} and \cite[Theorem 5.2]{Nak2}. Let us denote the remainder term in the prime geodesic theorem by $E_{\Gamma}(X)$.
Nakasuji showed that for $1\le T\le X^{1/2}$ we have
\begin{equation}\label{PrimeGeodesic to spec.sum}
E_{\Gamma}(X)=2\Re\left(\sum_{0<r_j\le T}\frac{X^{1+ir_j}}{1+ir_j}\right)+O\left(\frac{X^2}{T}\log X\right),
\end{equation}
where $\lambda_j=1+r_j^2$ are the eigenvalues of the hyperbolic Laplacian on $L^2(\Gamma  \setminus \HH^3)$.
Explicit formula  \eqref{PrimeGeodesic to spec.sum}  brings into play the spectral exponential sum
\begin{equation}\label{spec.exp.sum}
S(T,X)=\sum_{0<r_j\le T}X^{ir_j}.
\end{equation}

The trivial bound $$S(T,X)\ll T^{3}$$ follows from Weyl's law and yields Sarnak's result \eqref{error:Sarnak}. 

The proof of \eqref{error:CCL} is based on the following improvement (see \cite[Theorem 3.2]{ChatCherLaak})
\begin{equation}\label{CCL estimate on spec.sum}
S(T,X)\ll T^2X^{1/4}(TX)^{\epsilon}.
\end{equation}


The aim of this paper is to obtain a  new upper bound for the spectral exponential sum, which explicitly depends on  a subconvexity exponent $\theta$ for quadratic Dirichlet $L$-functions defined over Gaussian integers.
\begin{thm}\label{thm:spec.exp.sum new bound} For $1\le T\le X^{1/2}$ the following estimate holds
\begin{equation}\label{spec.exp.sum new bound}
S(T,X)\ll
X^{1/2+\theta/2}T(TX)^{\epsilon}.
\end{equation}
\end{thm}

As discussed in \cite[Remark 3.1]{ChatCherLaak}, it is not obvious what is the correct order of magnitude of $S(T,X)$ for all $X$ and $T$. In the two dimensional case, Petridis and Risager conjectured in  \cite{PR} that the spectral exponential sum exhibits square root cancelation, and Laaksonen verified this conjecture numerically.

The most important consequence of Theorem \ref{thm:spec.exp.sum new bound} is the new estimate on $E_{\Gamma}(X)$.
\begin{thm}\label{thmPrimeGeodesic new bound} The error term in the Prime Geodesic Theorem for the Picard manifold can be estimated as follows
\begin{equation}\label{PrimeGeodesic new bound}
E_{\Gamma}(X)\ll
X^{3/2+\theta/2+\epsilon}.
\end{equation}
\end{thm}

Assuming the Lindel\"{o}f hypothesis $\theta=0$, we improve the best known conditional result \eqref{est:koyama} and establish for the first time the error estimate $O(X^{3/2+\epsilon})$. Note that this is the best possible error admissible by the explicit formula of Nakasuji \eqref{PrimeGeodesic to spec.sum}. This is also an analogue of the best known conditional remainder term $O(X^{2/3+\epsilon})$ in the classical Prime Geodesic Theorem obtained by Iwaniec in \cite[p. 139]{IwPG} under the Lindel\"{o}f hypothesis for quadratic Dirichlet $L$-functions.



Theorem \ref{thm:spec.exp.sum new bound} is a consequence of the following estimate for the mean value of Maa{\ss} Rankin-Selberg $L$-functions on the critical line multiplied by 
the oscillating factor $X^{ir_j}$.

\begin{thm}\label{thm:momentoscill} Let $X\gg1$ and $X^{\epsilon}\le T\le X^{1/2}.$ Then for $s=1/2+it, |t|\ll T^{\epsilon}$ we have
\begin{equation}\label{symsquare estimate0}
\sum_{r_j}\frac{r_j}{\sinh(\pi r_j)}\omega_T(r_j)X^{ir_j}L(u_{j}\otimes u_{j},s)
\ll T^{3/2}X^{1/2+\theta+\epsilon},
\end{equation}
where $\omega_T(r_j)$ is a smooth characteristic function of the interval $(T,2T).$
\end{thm}

The standard method for investigating the left-hand side of \eqref{symsquare estimate0} is to estimate everything by  absolute value so that $X^{ir_j}$ is simply replaced by $1$ and the modulus of the Rankin-Selberg $L$-function is bounded by its square by the means of the Cauchy-Schwarz inequality. In this way, the problem is reduced to studying the second moment. This technique allows proving non-trivial results in certain ranges but the disadvantage is that the information about the behaviour of  $ X^{ir_j}$ is completely lost.

The main novelty and the core idea of our approach is that absolute value estimates are replaced by the method of analytic continuation. More precisely, we prove an exact formula for the left-hand side of  \eqref{symsquare estimate0} which allows us to take into consideration oscillations of the weight function $X^{ir_j}$. 
This approach has already proved to be effective in the two dimensional case (see \cite{BF2}), motivating us to develop the method further and to study the Prime Geodesic Theorem for the Picard manifold. We remark that in the three-dimensional setting, new technical difficulties arise, requiring the change of methodology and more involved analysis of special functions.

\

\section{Description of the problem}
In this section, following the book of Elsrodt, Grunewald and Mennicke \cite{EGM}, we provide some background information required for understanding the Prime Geodesic Theorem in the three-dimensional case. We keep all notations of \cite{EGM}.

The upper half space $\HH^3$ defined by \eqref{3dspace} is equipped with the {\it hyperbolic metric}
\begin{equation}\label{metric}
ds^2=\frac{dx^2+dy^2+dr^2}{r^2}.
\end{equation}

The associated {\it Laplace-Beltrami operator} is given by
\begin{equation*}
\Delta=r^2\left( \frac{\partial^2}{\partial x^2}+\frac{\partial^2}{\partial y^2}+\frac{\partial^2}{\partial r^2}\right)-r\frac{\partial}{\partial r}.
\end{equation*}

According to \cite[Proposition 1.6, p. 6]{EGM}, the {\it hyperbolic distance} $d(P,P')$ between two points  $P=z+rj$ and $P'=z+r'j$ of $\HH^3$ with $z,z' \in \C, \quad r,r'>0$
can be expressed as 
\begin{equation*}
\cosh{d(P,P')}=\frac{|z-z'|+r^2+r^{'2}}{2rr'}.
\end{equation*}


The matrix $$M=\begin{pmatrix}
a &b\\
c & d
\end{pmatrix}\in \PSL(2,\C)$$ 
acts on  $\HH^3$ as follows
\begin{equation}
M(z+rj)=z^*+r^*j, 
\end{equation}
where
\begin{equation}
z^*=\frac{(az+b)(\bar{c}\bar{z}+\bar{d})+a\bar{c}r^2}{|cz+d|^2+|c|^2r^2},\quad r^*=\frac{r}{|cz+d|^2+|c|^2r^2}.
\end{equation}

Elements $\gamma \in \mathbf{SL}(2,\C)$ such that $\gamma \neq \pm I$ can be classified  into four categories (see \cite[Definition 1.3, p. 34]{EGM}):
\begin{itemize}
\item
parabolic iff $|\mathbf{tr}(\gamma)|=2$ and $\mathbf{tr}(\gamma)\in \R$;
\item
hyperbolic iff $|\mathbf{tr}(\gamma)|>2$ and $\mathbf{tr}(\gamma)\in \R$;
\item
elliptic iff $0\leq  |\mathbf{tr}(\gamma)|<2$ and $\mathbf{tr}(\gamma)\in \R$;
\item
loxodromic otherwise.
\end{itemize}
 Note that elements of $\PSL(2, \C)$ are classified the same way as its preimages in $\mathbf{SL}(2,\C)$.
As shown in \cite[pp. 69-70]{EGM}, any hyperbolic or loxodromic element $T $ is conjugate in $\PSL(2,\C)$ to a unique element 
\begin{equation}\label{defat}
D(T)=\begin{bmatrix}
a(T) & 0\\
0 & a(T)^{-1}
\end{bmatrix}
, \quad 
|a(T)|>1,
\end{equation}
which acts on $\HH^3$ as follows
\begin{equation}
D(T)(z+rj)=K(T)z+N(T)rj, \quad z\in \C \text{, } r>0.
\end{equation}
Here $K(T)=a(T)^2$ is called the {\it multiplier} of $T$ and $N(T)=\left|a(T)\right|^2$ is called the {\it norm} of $T$.
Note that $K(T)$ and $N(T)$ depend only on the class of elements congugate to $T$, which we denote by $\{T\}$ .

Any hyperbolic or  loxodromic element $T$ has two distinct fixed points in the extended plane $\mathbf{P^1\C}$, and the geodesic connecting these two points is called the {\it axis} of $T$.
Axis are transformed into itself by the action of $T$.
According to \cite[p. 70]{EGM}, for all hyperbolic or loxodromic elements $T \in \PSL(2,\C)$ we have
\begin{equation}
\log{N(T)}=\inf\{ d(P,TP), \quad P \in \HH^3\},
\end{equation}
\begin{equation}
\log{N(T)}=d(P,TP) \text { iff $P$ is on the axis of } T.
\end{equation}

Let $\Gamma_d$ be a discrete subgroup of $\PSL(2, \C)$ and let $C(T)$ be a centralizer of a hyperbolic or loxodromic element $T$ in $\Gamma_d$.
We say that  $T_0$ is the {\it primitive hyperbolic or loxodromic element} associated to $T$ if $T_0 \in C(T)$ and it has minimal norm $N(T_0)>1$ among all elements of $C(T)$.  
As indicated in \cite[p. 192]{EGM}, $T$ determines uniquely the norm $N(T_0)$, but not  the primitive element $T_0$.

In this paper, we are mainly interested in the Picard group $\Gamma=\PSL(2, \Z[i])$, which acts discontinuously on $\HH^3$.
The quotient $\Gamma\setminus \HH^3$, called the {\it Picard manifold}, is a non-compact hyperbolic $3$-manifold of finite volume.
Primitive hyperbolic or loxodromic conjugacy classes in $\Gamma$ give rise to closed geodesics on $\Gamma\setminus \HH^3$( see \cite[p. 279]{Sarnak}), while norms of these conjugacy classes correspond to lengths of closed geodesics.

Accordingly, the Prime Geodesic Theorem for the Picard manifold is concerned with the asymptotic distribution of the norms of primitive hyperbolic and loxodromic elements of the Picard group. In other words, we are interested in  studying the asymptotic behaviour of the function
\begin{equation}
\pi_{\Gamma}(X)=\# \{\{T_0\}, \quad T_0 \in \Gamma \text{ primitive hyperbolic or loxodromic}, N(T_0)\leq X\}.
\end{equation}

As usual (see \cite[pp. 223-224]{EGM}), it is convenient to work with the summatory function
\begin{equation}
\Psi_{\Gamma}(X)=\sum_{\{T\}}\Lambda(T),
\end{equation}
where the summation is taken over all hyperbolic and loxodromic conjugacy classes in $\Gamma$ and 
\begin{equation}
\Lambda(T)=\frac{\log{N(T_0)}}{m(T)|a(T)-a(T)^{-1}|^2}
\end{equation}
with $m(T)$ being the order of the torsion of the centralizer $C(T)$, $T_0$ being a primitive element associated to $T$ and $a(T)$ being as in \eqref{defat}.

Following the approach of Iwaniec \cite{IwPG}, Nakasuji  showed in \cite[Theorem 4.1]{Nak} that for $1\le T\le X^{1/2}$ we have
\begin{equation}\label{PrimeGeodesic formula Nakasuji}
\Psi_{\Gamma}(X)=\frac{X^2}{2}+2\Re\left(\sum_{0<r_j\le T}\frac{X^{1+ir_j}}{1+ir_j}\right)+O\left(\frac{X^2}{T}\log X\right),
\end{equation}
where $\lambda_j=1+r_j^2$ are the eigenvalues of the hyperbolic Laplacian on $L^2(\Gamma  \setminus \HH^3)$. A detailed proof of \eqref{PrimeGeodesic formula Nakasuji} can be found in \cite[Theorem 5.2]{Nak2}. Note that the explicit formula \eqref{PrimeGeodesic formula Nakasuji} is obtained using the fact that there are no exceptional eigenvalues for the Picard group. This is a consequence of  Szmidt's result $\lambda_1 \geq \pi^2$ proved in \cite[Prop. 2, p. 397]{Szmidt}.

\section{Notation and preliminary results}

Throughout the paper we mostly use notations of \cite{Mot2001} that are slightly different from the standard ones (see \cite[Remark p. 270]{Mot2001}).

Let $\kk=\Q(i)$ be the Gaussian number field. All sums in this paper are over Gaussian integers unless otherwise indicated.

Let $\Gamma(z)$ be the Gamma function. According to the Stirling formula we have
\begin{equation}\label{Stirling1}
\Gamma(z)=\sqrt{2\pi}\exp(-z)z^{z-1/2}\left(1+O(z^{-1})\right)
\end{equation}
for $|z|\rightarrow\infty,|\arg(z)|<\pi.$ Note that instead of $O(z^{-1})$ it is possible to write arbitrarily accurate approximations of the Gamma function by evaluating more terms in the asymptotic expansion. As a consequence of \eqref{Stirling1}, we obtain
\begin{multline}\label{Stirling2}
\Gamma(\sigma+it)=\sqrt{2\pi}|t|^{\sigma-1/2}\exp(-\pi|t|/2)
\exp\left(i\left(t\log|t|-t+\frac{\pi t(\sigma-1/2)}{2|t|}\right)\right)\\\times
\left(1+O(|t|^{-1})\right)
\end{multline}
for $|t|\rightarrow\infty$ and a  fixed $\sigma$. Once again here instead of $O(|t|^{-1})$ it is possible to write an asymptotic expansion with as many terms as needed.

For convenience, we introduce the following notation
\begin{equation*}
\mathop{{\sum}^*}_{n=0}^{\infty}a_n=\frac{a_0}{2}+\sum_{n=1}^{\infty}a_n.
\end{equation*}
For a function $f(x)$, we denote its Mellin transform by
\begin{equation}\label{Mellin def}
\hat{f}(s)=\int_0^{\infty}f(x)x^{s-1}dx.
\end{equation}

For $\Re(s)>1$, the Dedekind zeta function is defined as $$\zeta_{\kk}(s)=4^{-1}\sum_{n\neq 0}|n|^{-2s}.$$
Let $\sigma_{\alpha}(n)=4^{-1}\sum_{d|n}|d|^{2\alpha}$. For $\Re(s)>1$ and $r\in \R$ we have (see \cite[p.403]{Szmidt})

\begin{equation}\label{sigma series}
\sum_{n\neq0}\frac{\sigma_{ir}^2(n)}{|n|^{2s+2ir}}=4\frac{\zeta_{\kk}(s+ir)\zeta^2_{\kk}(s)\zeta_{\kk}(s-ir)}{\zeta_{\kk}(2s)}.
\end{equation}

Let  $[n,x]=\Re(n\bar{x})$. We introduce two types of exponential functions
\begin{equation*}
e(x)=exp(2\pi ix)\quad\hbox{and}\quad e[x]=exp(2\pi i\Re(x)).
\end{equation*}
For $m,n,c\in \Z[i]$ with $c\neq 0$ define the Kloosterman sum
\begin{equation*}
S(m,n;c)=\sum_{\substack{a\pmod{c}\\ (a,c)=1}}e\left[m,\frac{a}{c}+n,\frac{a^*}{c}\right], \quad aa^*\equiv 1\pmod{c}.
\end{equation*}
This sum satisfies Weil's bound (see \cite[(3.5)]{Mot1997})

\begin{equation}\label{Weilbound}
|S(m,n;c)|\leq |c|\sigma_0(c)|(m,n,c)|.
\end{equation}
For $n,q \in\Z[i]$ and $q\neq0$ we have (see \cite[Lemma 1]{Mot2001})
\begin{equation}\label{linear exp sum}
\sum_{a\pmod{q}}e\left[\frac{an}{q}\right]=
\left\{
  \begin{array}{ll}
    |q|^2, & \hbox{if}\quad q|n, \\
    0, & \hbox{otherwise.}
  \end{array}
\right.
\end{equation}

\begin{lem}\label{Lerch lemma}
For $m\in\Z,\xi\in\C$ and $\Re(s)>1$ let
\begin{equation}\label{Lerch zeta}
\zeta_{\kk}(s;m,\xi)=\sum_{n+\xi\neq0}\left(\frac{n+\xi}{|n+\xi|}\right)^{m}\frac{1}{|n+\xi|^{2s}}.
\end{equation}
If $m\neq0$ then $\zeta_{\kk}(s;m,\xi)$ is entire in $s;$  otherwise it is regular in $s$ except for a simple pole at $s=1$ with residue $\pi.$ For $Re(s)<0$ we have
\begin{equation}\label{LerchFE}
\zeta_{\kk}(s;m,\xi)=(-i)^{|m|}\pi^{2s-1}\frac{\Gamma(1-s+|m|/2)}{\Gamma(s+|m|/2)}
\sum_{n\neq0}\left(\frac{n}{|n|}\right)^{m}\frac{e[n\bar{\xi}]}{|n|^{2(1-s)}}.
\end{equation}
\end{lem}
\begin{proof}
See \cite[Lemma 2]{Mot2001}.
\end{proof}


We denote by $\{\lambda_j=1+r_{j}^{2}$, $j=1,2,\ldots\}$ the non-trivial discrete spectrum of the hyperbolic Laplacian on $L^2(\Gamma  \setminus \HH^3)$, and by $\{u_j\}$ the corresponding orthonormal system of eigenfunctions.
Each function $u_j$ has a Fourier expansion of the form
\begin{equation}
u_j(z)=y\sum_{n\neq0}\rho_j(n)K_{ir_j}(2\pi|n|y)e[nx],
\end{equation}
where $K_{\nu}(z)$ is the K-Bessel function of order $\nu.$

The corresponding Rankin-Selberg $L$-function is defined as
\begin{equation}\label{RankinSelberg def}
L(u_j\otimes u_j,s)=\sum_{n\neq0}\frac{|\rho_j(n)|^2}{|n|^{2s}},\quad \Re{s}>1.
\end{equation}

Let us introduce the function
\begin{equation}
\J_{\nu}(z)=2^{-2\nu}|z|^{2\nu}J^{*}_{\nu}(z)J^{*}_{\nu}(\bar{z}),
\end{equation}
where $J^{*}_{\nu}(z)=J_{\nu}(z)(z/2)^{-\nu}$ and $J_{\nu}(z)$ is the $J$-Bessel function of order $\nu.$
Furthermore, we define the function
\begin{equation}\label{K-function def}
\K_{\nu}(z)=\frac{\J_{-\nu}(z)-\J_{\nu}(z)}{\sin(\pi\nu)}.
\end{equation}
\begin{thm}(Kuznetsov trace formula)
Let $h(r)$ be a holomorphic function in the region $|\Im(r)|<1/2+\epsilon$ such that
\begin{equation}\label{conditions on h}
h(r)=h(-r),\quad h(r)\ll(1+|r|)^{-3-\epsilon}
\end{equation}
for an arbitrary fixed $\epsilon>0.$ Then for any non-zero $m,n\in\Z[i]$
\begin{multline}\label{Kuznetsov formula}
\sum_{j=1}^{\infty}\frac{\overline{\rho_j(m)}\rho_j(n)}{\sinh(\pi r_j)}r_jh(r_j)+
2\pi\int_{-\infty}^{\infty}\frac{\sigma_{ir}(m)\sigma_{ir}(n)}{|mn|^{ir}|\zeta_{\kk}(1+ir)|^2}h(r)dr=\\
\frac{\delta_{m,n}+\delta_{m,-n}}{\pi^2}\int_{-\infty}^{\infty}r^2h(r)dr+
\sum_{q\neq0}\frac{S(m,n;q)}{|q|^2}\check{h}\left(\frac{2\pi\sqrt{mn}}{q}\right),
\end{multline}
where $\delta_{m,n}$ is the Kronecker delta and
\begin{equation}\label{K transform def}
\check{h}(z)=\frac{1}{2}\int_{-\infty}^{\infty}\K_{ir}(z)r^2h(r)dr.
\end{equation}
\end{thm}
\begin{proof}
This is \cite[Theorem 2]{Mot2001} except that in \cite{Mot2001} the integral transform \eqref{K transform def} is given in the alternative form
\begin{equation*}
\check{h}(z)=i\int_{-\infty}^{\infty}\frac{r^2}{\sinh(\pi r)}\J_{ir}(z)h(r)dr.
\end{equation*}
It can be shown that the two transforms are equal as follows
\begin{multline*}
i\int_{-\infty}^{\infty}\frac{r^2}{\sinh(\pi r)}\J_{ir}(z)h(r)dr=
\int_{0}^{\infty}i\frac{\J_{ir}(z)-\J_{-ir}(z)}{\sinh(\pi r)}r^2h(r)dr=
\int_{0}^{\infty}\K_{ir}(z)r^2h(r)dr.
\end{multline*}
Since $\K_{ir}(z)=\K_{-ir}(z)$, we obtain \eqref{K transform def}.
\end{proof}
\begin{lem}
Let $\vartheta=\arg(u).$ For $u\neq0$ and $|\Re(\nu)|<1/2$ we have
\begin{equation}\label{K-fun repsentation1}
\K_{\nu}(u)=\frac{8\cos(\pi\nu)}{\pi^2}\int_0^{\pi/2}\cos(2|u|\cos\vartheta\sin\tau)
K_{2\nu}(2|u|\cos\tau)d\tau,
\end{equation}
\begin{multline}\label{K-fun repsentation2}
\K_{\nu}(u)=\frac{8\cos(\pi\nu)}{\pi^2}
\mathop{{\sum}^*}_{m=0}^{\infty}(-1)^m
\left(\left(\frac{u}{|u|}\right)^{2m}+\left(\frac{u}{|u|}\right)^{-2m}\right)\\
\times\int_0^{\pi/2}J_{2m}(2|u|\sin\tau)K_{2\nu}(2|u|\cos\tau)d\tau.
\end{multline}
\end{lem}
\begin{proof}
First, note that $\K_{\nu}(u)=\K_{\nu,0}(u)$, where $\K_{\nu,0}(u)$ is defined in \cite[(6.21), (7.21)]{BrugMot}. Then \eqref{K-fun repsentation1} and \eqref{K-fun repsentation2} follow from \cite[(12.16)]{BrugMot}. More precisely, \eqref{K-fun repsentation1} is a consequence of \cite[Lemma 8]{Mot1997} and \eqref{K-function def}. To prove \eqref{K-fun repsentation2}, we use the relation (see \cite[10.12.2]{HMF})
\begin{equation*}
\cos(2|u|\cos\vartheta\sin\tau)=\sum_{m\in\Z}(-1)^me(m\vartheta)J_{2m}(2|u|\sin\tau).
\end{equation*}
Then it follows from \eqref{K-fun repsentation1} that
\begin{equation*}
\K_{\nu}(u)=\frac{8\cos(\pi\nu)}{\pi^2}
\sum_{m=-\infty}^{\infty}(-1)^me(m\vartheta)
\int_0^{\pi/2}J_{2m}(2|u|\sin\tau)K_{2\nu}(2|u|\cos\tau)d\tau.
\end{equation*}
Since (see \cite[10.4.1]{HMF}) $J_{2m}(z)=J_{-2m}(z)$, we obtain \eqref{K-fun repsentation2}.
\end{proof}
We introduce the $L$-function
\begin{equation}\label{L beatiful def}
\mathscr{L}_{\kk}(s;n)=\frac{\zeta_{\kk}(2s)}{\zeta_{\kk}(s)}\sum_{q\neq0}\frac{\rho_q(n)}{|q|^{2s}},\quad Re{s}>1,
\end{equation}
\begin{equation}\label{rho def}
\rho_q(n):=\#\{x\Mod{2q}:x^2\equiv n\Mod{4q}\}.
\end{equation}
Note that solutions of the congruence are counted in the ring of the Gaussian integers. 

Generalising the real case studied by Zagier in \cite[Proposition 3]{Z}, Szmidt investigated properties of the $L$-function \eqref{L beatiful def} in \cite{Szmidt}. Some of these properties are listed below.
Note that $\mathscr{L}_{\kk}(s;n)=4\zeta(s,n)/\zeta_{\kk}(s)$ in the notations of \cite{Szmidt}. 

The $L$-function $\mathscr{L}_{\kk}(s;n)$ does not vanish only for $n \equiv 0,1 \Mod{4}.$   
If $n=0$ then
\begin{equation}\label{L beautiful 0}
\mathscr{L}_{\kk}(s;0)=4\zeta_{\kk}(2s-1).
\end{equation}
Otherwise, for $n=Dl^2$ ($D$ is the discriminant of the corresponding extension of $\kk$) we have

\begin{equation}\label{ldecomp}
\mathscr{L}_{\kk}(s;n)=4T_{l}^{(D)}(s)L(s,\chi_D),
\end{equation}
where
\begin{equation}\label{eq:td}
T_{l}^{(D)}(s)=\frac{1}{4}\sum_{d|l}\chi_D(d)\mu(d)|d|^{-2s}\sigma_{1-2s}\left(\frac{l}{d}\right)
\end{equation}
and
\begin{equation}
L(s,\chi_D)=\frac{1}{4}\sum_{n\neq0}\frac{\chi_D(n)}{|n|^{2s}}
\end{equation}
with $\chi_D(n)=(D/n)$ being the corresponding Kronecker symbol for the quadratic extension of $\kk.$ 

Let $\theta$ denote the subconvexity exponent for  $L(1/2+it,\chi_D)$ so that
\begin{equation}\label{subconvexity}
L(1/2+it,\chi_D)\ll (1+|t|)^A|D|^{\theta}.
\end{equation}

According to the results of Michel-Venkatesh \cite{MV}, Wu \cite{Wu} and Nakasuji \cite{Nak3}, we can take  $$\theta=A=1/2-(1-2\alpha)/8+\epsilon, \quad \alpha=7/64.$$

It follows from  \eqref{ldecomp},  \eqref{eq:td} and \eqref{subconvexity}   that
\begin{equation}\label{Lbeaut subconvex}
\mathscr{L}_{\kk}(1/2+it;n)\ll (1+|t|)^A|n|^{\theta+\epsilon}.
\end{equation}

Next, we prove an analogue of \cite[Lemma 4.1]{BF1}, relating sums of Kloosterman sums and the $L$-function \eqref{L beatiful def}.
\begin{lem}
The following identity holds
\begin{equation}\label{eq:sumofklsums}
\sum_{q\neq0}\frac{1}{|q|^{2+2s}}\sum_{c\Mod{q}}S(c,c;q)e\left[n\overline{c/q}\right]=\frac{\zeta_{\kk}(s)}{\zeta_{\kk}(2s)}
\mathscr{L}_{\kk}(s;\bar{n}^2-4).
\end{equation}
\end{lem}
\begin{proof}
Since $\Re(a\bar{b})=\Re(\bar{a}b)$, it follows that $e[n\overline{c/q}]=e[\bar{n}c/q].$ Consider
\begin{equation*}
S:=\sum_{c\Mod{q}}S(c,c;q)e[n\overline{c/q}]=
\sum_{c\Mod{q}}\sum_{\substack{a \Mod{q}\\(a,q)=1}}e\left[ \frac{(a+a^*+\bar{n})c}{q}\right],
\end{equation*}
where $aa^*\equiv 1\Mod{q}$. Using \eqref{linear exp sum}, we obtain
\begin{equation*}
S=
\sum_{\substack{a \Mod{q},(a,q)=1\\ a+a^*+\bar{n}\equiv 0 \Mod{q}}}|q|^2=
\sum_{\substack{a \Mod{q},\\ a^2+a\bar{n}+1\equiv 0 \Mod{q}}}|q|^2.
\end{equation*}
Similarly to the real case, one can prove that there is one-to-one correspondence between the solutions $a \Mod{q}$ of $a^2+a\bar{n}+1\equiv 0 \Mod{q}$ and  the solutions $b \Mod{2q}$ of $b^2\equiv \bar{n}^2-4 \Mod{4q}$. Thus using \eqref{rho def} we have
\begin{equation*}
S=\sum_{\substack{b\Mod{2q}\\ b^2\equiv \bar{n}^2-4\Mod{4q}}}|q|^2=|q|^2\rho_q(\bar{n}^2-4).
\end{equation*}
Finally,
\begin{multline*}
\sum_{q\neq0}\frac{1}{|q|^{2+2s}}\sum_{c\Mod{q}}S(c,c;q)e\left[n\overline{c/q}\right]=
\sum_{q\neq0}\frac{\rho_q(\bar{n}^2-4)}{|q|^{2s}}=
\frac{\zeta_{\kk}(s)}{\zeta_{\kk}(2s)}
\mathscr{L}_{\kk}(s;\bar{n}^2-4).
\end{multline*}
\end{proof}


\section{Exact formula for the first moment of Maa{\ss} Rankin-Selberg $L$-functions}

In this section we prove an exact formula for first moment
\begin{equation}\label{firstmomentdef}
M_1(s):=\sum_{j}\frac{r_j}{\sinh(\pi r_j)}h(r_j)L(u_{j}\otimes u_{j},s),
\end{equation}
where $h(r)$ is an even function, holomorphic in any fixed horizontal strip and satisfying the conditions
\begin{equation}\label{conditions on h 1}
h(\pm(n-1/2)i)=0,\quad h(\pm ni)=0\quad\hbox{for}\quad n=1,2\ldots,N,
\end{equation}
\begin{equation}\label{conditions on h 2}
h(r)\ll\exp(-c|r|^2)
\end{equation}
for some fixed $N$ and $c>0.$

We remark that our approach here differs from the one we used in the two dimensional case in \cite{BF2}. 
Our method incorporates some ideas from \cite{BF1} and from  Motohashi's proof \cite{Mot1992} of the exact formula for the second moment of classical Maa{\ss} $L$-functions. 

\subsection{Exact formula for large $s$}
 For $\Re{s}>3/2$ we define
\begin{equation}\label{sigmadef}
\Sigma(s)=\sum_{n\neq0}\frac{1}{|n|^{2s}}\sum_{q\neq0}\frac{S(n,n;q)}{|q|^2}\check{h}\left(\frac{2\pi n}{q}\right),
\end{equation}
where the transform $\check{h}(z)$ is given by \eqref{K transform def}.

First of all, we identify the region of convergence of \eqref{sigmadef}.
\begin{lem}\label{lem abs convergent}
For $\Re{s}>3/2$ the double series \eqref{sigmadef} converges absolutely.
\end{lem}
\begin{proof}
Convergence of the double series \eqref{sigmadef} for $\Re{s}>3/2$ follows from \eqref{Weilbound} and the estimate
\begin{equation}\label{estimate on transform of h}
\check{h}(u)\ll\min(|u|^{3/2-\epsilon},|u|^{\epsilon}).
\end{equation}
To prove \eqref{estimate on transform of h} we proceed as follows.  According to \eqref{K transform def} and \eqref{K-fun repsentation1} we have
\begin{equation*}
\check{h}(u)=\frac{4}{\pi^2}\int_0^{\pi/2}\cos(2|u|\cos\vartheta\sin\tau)
\int_{-\infty}^{\infty}r^2h(r)\cosh(\pi r)K_{2ir}(2|u|\cos\tau)drd\tau,
\end{equation*}
where $\vartheta=\arg(u).$ Following the arguments of Motohashi \cite[(2.11),(2.12)]{Mot1992}, we obtain for $0<a<1/2$
\begin{equation}\label{K transform1}
\check{h}(u)=\frac{1}{\pi^2}\int_0^{\pi/2}\cos(2|u|\cos\vartheta\sin\tau)
\int_{(a)}\frac{h^{*}(s)}{\cos(\pi s)}(|u|\cos\tau)^{-s}dsd\tau,
\end{equation}
where
\begin{equation}\label{h star}
h^{*}(s)=\int_{-\infty}^{\infty}r^2h(r)\coth(\pi r)\frac{\Gamma(s+ir)}{\Gamma(1-s+ir)}dr.
\end{equation}
Due to the fact that $h(r)$ satisfies \eqref{conditions on h 1}, we can move the line of integration in \eqref{h star} to $\Im(r)=-C$ with $0<C<N+1$. This yields the analytic continuation of $h^{*}(s)$ to  the region $\Re(s)>-C$. Since the function $h(r)$ is even, we obtain by repeating Motohashi's proof of \cite[(2.14)]{Mot1992} that
\begin{equation}\label{h star vanish}
h^{*}(\pm i/2)=0.
\end{equation}
Thus we can move the line of integration in \eqref{K transform1} to $a=-3/2+\epsilon$ without crossing any pole. The rapid decay of $h(r)$, as indicated by \eqref{conditions on h 2}, implies that if $\Re(s)$ is bounded, then $h^{*}(s)\ll |\Im(s)|^A$ for a sufficiently large $A$.  Therefore, $\check{h}(u)\ll |u|^{3/2-\epsilon}.$ Similarly, moving the line of integration in \eqref{K transform1} to $a=\epsilon$, we show that $\check{h}(u)\ll |u|^{\epsilon}.$
\end{proof}

The next step is to express the moment in terms of the double series \eqref{sigmadef}.
\begin{lem} For $\Re{s}>3/2$ we have
\begin{multline}\label{firstmoment=MT+Sigma-ContSpectr}
M_1(s)=
\frac{4\zeta_{\kk}(s)}{\pi^2}\int_{-\infty}^{\infty}r^2h(r)dr+\Sigma(s)
-8\pi\frac{\zeta^2_{\kk}(s)}{\zeta_{\kk}(2s)}\int_{-\infty}^{\infty}
\frac{\zeta_{\kk}(s+ir)\zeta_{\kk}(s-ir)}{\zeta_{\kk}(1+ir)\zeta_{\kk}(1-ir)}h(r)dr.
\end{multline}
\end{lem}
\begin{proof}
Applying the Kuznetsov trace formula \eqref{Kuznetsov formula} and using \eqref{sigma series}, we prove the lemma.
\end{proof}

\subsection{Technical preparation}

For $\Re{s}>3/2$ the absolute convergence (see Lemma \ref{lem abs convergent}) allows us to change the order of summation in \eqref{sigmadef}. Using \eqref{K transform def} and \eqref{K-fun repsentation2}, we obtain
\begin{multline}\label{Sigma2}
\Sigma(s)=\frac{8}{\pi^2}
\sum_{q\neq0}\frac{1}{|q|^2}\sum_{n\neq0}\frac{S(n,n;q)}{|n|^{2s}}
\mathop{{\sum}^*}_{m=0}^{\infty}(-1)^m\\\times
\left(\left(\frac{n|q|}{|n|q}\right)^{2m}+\left(\frac{n|q|}{|n|q}\right)^{-2m}\right)
\int_0^{\pi/2}\psi\left(m,\tau;\frac{2\pi |n|}{|q|}\right)d\tau,
\end{multline}
where for a positive integer $m$,  $0<\tau<\pi/2$ and $r\in \R$
\begin{equation}\label{psi def}
\psi(m,\tau;z)=\frac{1}{2}\int_{-\infty}^{\infty}r^2h(r)\cosh(\pi r)g(m,r,\tau;z)dr,
\end{equation}
\begin{equation}\label{g def}
g(m,r,\tau;z)=J_{2m}(2z\sin\tau)K_{2ir}(2z\cos\tau).
\end{equation}

This subsection is devoted to the technical preparation for the proof of the analytic continuation of \eqref{Sigma2} to the region containing $\Re{s}=1/2$.
In particular, we investigate various properties of the function $\psi(m,\tau;z)$ defined by \eqref{psi def}.

\begin{lem}\label{Mellin lemma} For $0>a>-2m-2N-2$ we have
\begin{equation}\label{Mellin for psi}
\psi(m,\tau;z)=\frac{1}{2\pi i}\int_{(a)}\frac{h^{*}(m,\tau;s)}{\cos(\pi s/2)}(z\cos\tau)^{-s}ds,
\end{equation}
where
\begin{multline}\label{h star def}
h^{*}(m,\tau;s)=2\pi i\frac{(-1)^m\sin^m\tau}{2^{4+m}\cos^{2m}\tau}
\int_{-\infty}^{\infty}\frac{r^2h(r)\coth(\pi r)\Gamma(m+s/2+ir)}{\Gamma(1-m-s/2+ir)\Gamma(1+2m)}\\
\times F(m+s/2+ir,m+s/2-ir,1+2m;-\tan^2\tau)dr.
\end{multline}
\end{lem}
\begin{proof}
For $\Re(s)>-2m$, using \cite[6.576.3]{GR}, we compute the Mellin transform of the function $g(m,r,\tau;z)$ with respect to $z$
\begin{multline}\label{g Mellin}
\hat{g}(m,r,\tau;s)=\frac{\sin^m\tau}{2^{2+m}\cos^{2m+s}\tau}
\frac{\Gamma(m+s/2+ir)\Gamma(m+s/2-ir)}{\Gamma(1+2m)}\\
\times F(m+s/2+ir,m+s/2-ir,1+2m;-\tan^2\tau).
\end{multline}
Applying the Mellin inversion formula, we obtain for $a>-2m$
\begin{multline}\label{Mellin for psi2}
\psi(m,\tau;z)=\frac{1}{2}\int_{-\infty}^{+\infty}r^2h(r)\cosh{(\pi r)}\frac{1}{2\pi i}\int_{(a)}\hat{g}(m,r,\tau;s)z^{-s}dsdr\\=
\frac{1}{2\pi i}\Biggl(\frac{1}{2}\int_{-\infty}^{+\infty} r^2h(r)\cosh{(\pi r)}\hat{g}(m,r,\tau;s)dr\Biggr)z^{-s}ds
\\ =
\frac{1}{2\pi i}\int_{(a)}
\frac{h^{*}(m,\tau;s)}{\cos(\pi s/2)}(z\cos\tau)^{-s}ds,
\end{multline}
where
\begin{multline}\label{h star def2}
h^{*}(m,\tau;s)=\frac{\sin^m\tau\cos(\pi s/2)}{2^{3+m}\cos^{2m}\tau}
\int_{-\infty}^{\infty}r^2h(r)\cosh(\pi r)
\frac{\Gamma(m+s/2+ir)\Gamma(m+s/2-ir)}{\Gamma(1+2m)}\\
\times F(m+s/2+ir,m+s/2-ir,1+2m;-\tan^2\tau)dr.
\end{multline}
Using \cite[(2.10)]{Mot1992} we have
\begin{multline}\label{Gamma transform}
\Gamma(m+s/2+ir)\Gamma(m+s/2-ir)=
\frac{(-1)^m\pi i}{2\sinh(\pi r)\cos(\pi s/2)}\\
\times\left(\frac{\Gamma(m+s/2+ir)}{\Gamma(1-m-s/2+ir)}-\frac{\Gamma(m+s/2-ir)}{\Gamma(1-m-s/2-ir)}\right).
\end{multline}
Substituting \eqref{Gamma transform} to \eqref{h star def2}, we obtain \eqref{h star def}. In view of \eqref{conditions on h 1}, we can move the line of integration in \eqref{h star def} to $\Im(r)=-C$ with $0<C<N+1$. This yields the analytic continuation of the function $h^{*}(m,\tau;s)$ to the region $\Re(s)>-2m-2N-2.$ To be able to move the line of integration in \eqref{Mellin for psi2} to $a>-2m-2N-2$, and thus to prove the lemma, it is sufficient to show that
\begin{equation}\label{h star zero2}
h^{*}(m,\tau;\pm(1+2j))=0\quad\hbox{for}\quad 0\le j\le N+m.
\end{equation}
Let us consider the case $-(1+2j).$ Accordingly,
\begin{multline}\label{h star j}
h^{*}(m,\tau;-(1+2j))=\frac{(-1)^m\sin^m\tau}{2^{4+m}\cos^{2m}\tau}2\pi i\\\times
\int_{\Im(r)=-C}\frac{r^2h(r)\coth(\pi r)\Gamma(m-1/2-j+ir)}{\Gamma(1-m+1/2+j+ir)\Gamma(1+2m)}\\
\times F(m-1/2-j+ir,m-1/2-j-ir,1+2m;-\tan^2\tau)dr.
\end{multline}
Using the property \eqref{conditions on h 1}, we move the line of integration in \eqref{h star j} back to $\Im(r)=0$  and remark that the function under the integral sign is odd. This proves \eqref{h star zero2}.
\end{proof}


The next lemma guarantees the convergence of the integral \eqref{Mellin for psi} for all bounded $a$ if $m$ is a constant and for $a<0$ if $m\rightarrow\infty.$ In fact, it is possible to prove the convergence of the integral \eqref{Mellin for psi} for all bounded $a$ when $m\rightarrow\infty.$ However, this is not required for our purposes.


\begin{lem}\label{hstar-estimate}
Let $s/2=\sigma+it.$  Suppose that $m$ is a positive integer and that the number $0<\tau<\pi/2$ is fixed. For  fixed $m$ and $\sigma$  we have
\begin{equation}\label{hstar-estimate small m}
\frac{h^{*}(m,\tau;s)}{\cos(\pi s/2)}\ll(1+|t|)^{-A}\text{ for any }A>1.
\end{equation}

For a fixed $\sigma$ and $m\rightarrow\infty$ we have
\begin{equation}\label{hstar-estimate big m}
\frac{h^{*}(m,\tau;s)}{\cos(\pi s/2)}\ll B^{-m}(1+|t|)^{2\sigma-1},
\end{equation}
where $B>1$ is some constant.
\end{lem}
\begin{proof}
Let $y=\tan^2\tau.$ We first consider the case of fixed $m.$ Applying \cite[15.6.6]{HMF} we have
\begin{multline}\label{2F1 Mellin1}
F(m+s/2+ir,m+s/2-ir,1+2m;-y)=\\
\frac{\Gamma(1+2m)}{\Gamma(m+s/2+ir)\Gamma(m+s/2-ir)}
J(y,m,r,s)
\end{multline}
with
\begin{multline}\label{2F1 Mellin J}
J(y,m,r,s)=\frac{1}{2\pi i}\int_{(l)}\frac{\Gamma(m+s/2+ir+z)\Gamma(m+s/2-ir+z)\Gamma(-z)}{\Gamma(1+2m+z)}y^{z}dz,
\end{multline}
where $l$ is a line that separates the poles of $\Gamma(m+s/2\pm ir+z)$ and $\Gamma(-z).$ Substituting \eqref{2F1 Mellin1} to \eqref{h star def}, we obtain
\begin{multline}\label{hstar-estimate1}
\frac{h^{*}(m,\tau;s)}{\cos(\pi s/2)}\ll
\left(\frac{y(1+y)}{4}\right)^{m/2}\\\times
\int_{-\infty}^{\infty}
\frac{r^2h(r)\coth(\pi r)|J(y,m,r,s)|dr}{|\Gamma(1-m-s/2+ir)\Gamma(m+s/2-ir)\cos(\pi s/2)|}.
\end{multline}
Next, we apply \cite[5.5.3]{HMF}, getting
\begin{multline}\label{hstar-estimate2}
\frac{h^{*}(m,\tau;s)}{\cos(\pi s/2)}\ll
\left(\frac{y(1+y)}{4}\right)^{m/2}\\\times
\int_{-\infty}^{\infty}
r^2h(r)\coth(\pi r)\exp(-\pi(|t-r|-|t|))|J(y,m,r,s)|dr.
\end{multline}
For a fixed $m$ we can ignore the multiple $(y(1+y)/4)^{m/2}.$ The property \eqref{conditions on h 2} for the function $h(r)$ allows us to restrict the integral over $r$ to the neighborhood of zero. Thus it is sufficient to consider only $|r|\ll |t|^{\epsilon}. $  In order to estimate $J(y,m,r,s)$ for such $r$, we move the contour of integration in \eqref{2F1 Mellin J} to the right of the line $\Re(z)=c$ such that $m+\sigma+c>0$. Making the change of variable $z=c+iv$ in the integral, we obtain
\begin{multline}\label{2F1 Mellin2}
J(y,m,r,s)=\sum_{0\le j\le c}\frac{(-1)^jy^{j}\Gamma(m+s/2+ir+j)\Gamma(m+s/2-ir+j)}{j!\Gamma(1+2m+j)}+\\
O\left(\int_{-\infty}^{\infty}
y^{c}\frac{\Gamma(m+\sigma+c+i(t+r+v))\Gamma(m+\sigma+c+i(t-r+v))\Gamma(-c-iv)}{\Gamma(1+2m+c+iv)}dv
\right).
\end{multline}
It follows from \eqref{Stirling2} that the main contribution in the $v$ integral comes from the interval $(-t-|r|,-t+|r|).$ To estimate the sum over $j$ in \eqref{2F1 Mellin2}, we again use \eqref{Stirling2} and prove for $|r|\ll |t|^{\epsilon}$ the estimate
\begin{equation}\label{2F1 Mellin3}
J(y,m,r,s)\ll
\exp(-\pi|r|)(1+|t|)^{-1-2m-2c+\epsilon}.
\end{equation}
Substituting \eqref{2F1 Mellin3} to \eqref{hstar-estimate2}, we obtain \eqref{hstar-estimate big m}.
\par
Let us consider the case of $m\rightarrow\infty.$ In that case we cannot ignore the multiple $(y(1+y)/4)^{m/2}$ in \eqref{hstar-estimate2}. Consequently, we move the contour of integration in \eqref{2F1 Mellin J} to the left of the line $\Re(z)=-m+c$ without crossing any pole. Making the change of variables $z=-m+c+iv$, we obtain
\begin{multline}\label{2F1 Mellin4}
J(y,m,r,s)\ll \\ y^{-m}
\int_{-\infty}^{\infty}
\frac{\Gamma(\sigma+c+i(t+r+v))\Gamma(\sigma+c+i(t-r+v))\Gamma(m-c-iv)}{\Gamma(1+m+c+iv)}dv.
\end{multline}
Note that the formula \eqref{Stirling2} shows that the main contribution in the $v$ integral comes once again from the interval $(-t-|r|,-t+|r|).$ Since
\begin{equation*}
\Gamma(m-c-iv)/\Gamma(1+m+c+iv)\ll(m+|v|)^{-1-2c},
\end{equation*}
 we obtain for $|r|\ll |t|^{\epsilon}$ the estimate
\begin{equation}\label{2F1 Mellin5}
J(y,m,r,s)\ll y^{-m}\exp(-\pi|r|)(m+|t|)^{-1-2c+\epsilon}.
\end{equation}
Application of \eqref{2F1 Mellin5} and \eqref{hstar-estimate2}  yields
\begin{equation}\label{hstar-estimate3}
\frac{h^{*}(m,\tau;s)}{\cos(\pi s/2)}\ll
\left(\frac{(1+y)}{4y}\right)^{m/2}(m+|t|)^{-A}
\end{equation}
for any constant $A>1$. This implies \eqref{hstar-estimate big m} as long as $y>1/3.$  

The case $y<1/3$ can be treated similarly but we prefer to follow a simpler approach based on the Euler integral representation formula \cite[15.6.1]{HMF} for the hypergeometric function in \eqref{h star def}.  Applying this formula, we obtain
\begin{multline}\label{hstar-estimate4}
\frac{h^{*}(m,\tau;s)}{\cos(\pi s/2)}\ll
\left(\frac{y(1+y)}{4}\right)^{m/2}\\\times
\int_{-\infty}^{\infty}
\frac{r^2h(r)\coth(\pi r)|E(y,m,r,s)|dr}{|\Gamma(1-m-s/2+ir)\Gamma(m+1-s/2-ir)\cos(\pi s/2)|},
\end{multline}
where
\begin{equation}\label{E integral1}
E(y,m,r,s)=\int_0^1\frac{x^{m+s/2-1+ir}(1-x)^{m-s/2-ir}}{(1+xy)^{m+s/2-ir}}dx.
\end{equation}
Next, we estimate the  integral above by absolute value and use \cite[5.5.5]{HMF} to show that
\begin{multline}\label{E integral2}
E(y,m,r,s)\ll \int_0^1x^{m+\sigma-1}(1-x)^{m-\sigma}dx=\frac{\Gamma(m+\sigma)\Gamma(m-\sigma+1)}{\Gamma(2m+1)}\\
\ll m^{-1/2}2^{-2m}.
\end{multline}
Substituting \eqref{E integral2} to \eqref{hstar-estimate4} and applying \cite[5.5.3]{HMF} we have
\begin{multline}\label{hstar-estimate5}
\frac{h^{*}(m,\tau;s)}{\cos(\pi s/2)}\ll
\left(\frac{y(1+y)}{64}\right)^{m/2}\\\times
\int_{-\infty}^{\infty}
\frac{\Gamma(m+s/2-ir)}{\Gamma(m+1-s/2-ir)}
r^2h(r)\coth(\pi r)\exp(\pi(|t-r|-|t|))dr.
\end{multline}
Using the Stirling formula  \eqref{Stirling1} and the property \eqref{conditions on h 2} of the function $h(r)$, we finally obtain the estimate
\begin{equation}\label{hstar-estimate6}
\frac{h^{*}(m,\tau;s)}{\cos(\pi s/2)}\ll
\left(\frac{y(1+y)}{64}\right)^{m/2}(1+|t|)^{2\sigma-1},
\end{equation}
which proves \eqref{hstar-estimate big m}.
\end{proof}
\begin{cor}  For a positive integer  $m$ we have
\begin{equation}\label{psi estimate}
\psi(m,\tau;z)\ll B^{-m}\min(|z|^{A},|z|^{\epsilon}),
\end{equation}
where $A$ and $B$ are some constants such that $0<A<2m+2N+2$ and $B>1$.
\end{cor}
\begin{proof}
This is a direct consequence of \eqref{Mellin for psi} and Lemma \ref{hstar-estimate}.
\end{proof}

\subsection{Analytic continuation to the region $\Re{s}>1/2$}

For $n\neq0,\,0<\tau<\pi/2$ and $\Re(s)>1/2$ define
\begin{multline}\label{I integral def}
I(n,\tau,s):=|n|^{2s-2}2^{1-2s}\mathop{{\sum}^*}_{m=0}^{\infty}
\left(\left(\frac{n}{|n|}\right)^{2m}+\left(\frac{n}{|n|}\right)^{-2m}
\right)\\\times
\frac{1}{2\pi i}\int_{(a)}\frac{h^{*}(m,\tau;w)}{\cos(\pi w/2)}(\cos\tau)^{-w}
\frac{\Gamma(1-s-w/2+m)}{\Gamma(s+w/2+m)}
\left(\frac{|n|}{2}\right)^wdw,
\end{multline}
where $a<2(1-\Re(s)+m).$ Note that Lemma \ref{hstar-estimate} and the Stirling formula \eqref{Stirling2} show that the sum over $m$ and the integral over $w$ are absolutely convergent provided that $\Re(s)>1/2.$ For $0<\tau<\pi/2$ and $\Re(s)>1/2$ define
\begin{equation}\label{I0 integral def}
I(0,\tau,s):=
-\frac{h^{*}(0,\tau;2-2s)}{2(\cos\tau)^{2-2s}\cos(\pi s)}.
\end{equation}


\begin{lem}\label{Sigma Lemma}  For $\Re(s)>1/2$ we have
\begin{multline}\label{Sigma1}
\Sigma(s)=
\frac{8(2\pi)^{2s-1}}{\pi^2}
\int_0^{\pi/2}
\frac{\zeta_{\kk}(s)}{\zeta_{\kk}(2s)}\mathscr{L}_{\kk}(s;-4)
I(0,\tau,s)d\tau+\\
\frac{8(2\pi)^{2s-1}}{\pi^2}
\int_0^{\pi/2}
\frac{\zeta_{\kk}(s)}{\zeta_{\kk}(2s)}\sum_{n\neq0}\mathscr{L}_{\kk}(s;\bar{n}^2-4)
I(n,\tau,s)d\tau.
\end{multline}

\end{lem}
\begin{proof}
First, we assume that $\Re{s}>3/2.$  Then we can apply \eqref{Sigma2}, namely
\begin{multline*}
\Sigma(s)=\frac{8}{\pi^2}
\sum_{q\neq0}\frac{1}{|q|^2}\sum_{n\neq0}\frac{S(n,n;q)}{|n|^{2s}}
\mathop{{\sum}^*}_{m=0}^{\infty}(-1)^m\\\times
\left(\left(\frac{n|q|}{|n|q}\right)^{2m}+\left(\frac{n|q|}{|n|q}\right)^{-2m}\right)
\int_0^{\pi/2}\psi\left(m,\tau;\frac{2\pi |n|}{|q|}\right)d\tau.
\end{multline*}
It is required to verify that all series and integrals in the formula above are absolutely convergent. The only difference with Lemma \ref{lem abs convergent} is the presence of the sum over $m.$ To analyse this sum, it is sufficient to use standard bounds for the $J$-Bessel function. See, for example, \cite[Lemma 3.5]{Hough}, where all required estimates are presented. It is also possible to use the estimate \eqref{psi estimate}.

 It follows from  \eqref{Mellin for psi} that
\begin{multline}\label{Sigma3}
\Sigma(s)=\frac{8}{\pi^2}
\int_0^{\pi/2}
\mathop{{\sum}^*}_{m=0}^{\infty}(-1)^m
\frac{1}{2\pi i}\int_{(a)}\frac{h^{*}(m,\tau;w)}{\cos(\pi w/2)}(\cos\tau)^{-w}\\\times
\sum_{q\neq0}\frac{1}{|q|^2}\sum_{n\neq0}\frac{S(n,n;q)}{|n|^{2s}}
\left(\left(\frac{n|q|}{|n|q}\right)^{2m}+\left(\frac{n|q|}{|n|q}\right)^{-2m}\right)
\left(\frac{2\pi |n|}{|q|}\right)^{-w}dwd\tau,
\end{multline}
where $0>a>2-2\Re(s).$ The next step is to apply \eqref{Lerch zeta}, which gives
\begin{multline}\label{Kloosterman to Lerch}
\sum_{n\neq0}\frac{S(n,n;q)}{|n|^{2s+w}}\left(\frac{n}{|n|}\right)^{2m}=
\left(\frac{q}{|q|}\right)^{2m}\sum_{c\Mod{q}}\frac{S(c,c;q)}{|q|^{2s+w}}\zeta_{\kk}(s+w/2;2m,c/q).
\end{multline}
Substituting \eqref{Kloosterman to Lerch} to \eqref{Sigma3}, we obtain
\begin{multline}\label{Sigma4}
\Sigma(s)=\frac{8}{\pi^2}
\int_0^{\pi/2}
\mathop{{\sum}^*}_{m=0}^{\infty}(-1)^m
\frac{1}{2\pi i}\int_{(a)}\frac{h^{*}(m,\tau;w)}{\cos(\pi w/2)}(\cos\tau)^{-w}
\sum_{q\neq0}\frac{1}{|q|^{2+2s}}
\sum_{c\Mod{q}}S(c,c;q)\\\times
\left(\zeta_{\kk}(s+w/2;2m,c/q)+\zeta_{\kk}(s+w/2;-2m,c/q)\right)
\frac{dwd\tau}{(2\pi)^w}.
\end{multline}
According to Lemma \ref{Lerch lemma}, the Lerch zeta function has a simple pole at $w=2-2s$ only for $m=0$ (but for all $c$) with residue $\pi$. Thus moving the $w$-contour to the left up to $a_1:=-2s-\epsilon$, we cross these poles. We remark that Lemma \ref{Mellin lemma} plays a crucial  role here since it enables us to move the contour up to $a_1>-2m-2N-2$. Note that $N=0$ is not sufficient!

The final step is to apply the functional equation \eqref{LerchFE} for the Lerch zeta function,  which is now possible since $\Re(s+w/2)<0$. Using \eqref{LerchFE} and \eqref{eq:sumofklsums}, we obtain
\begin{multline}\label{Sigma5}
\sum_{q\neq0}\frac{1}{|q|^{2+2s}}
\sum_{c\Mod{q}}S(c,c;q)
\zeta_{\kk}(s+w/2;2m,c/q)=\\
(-i)^{|2m|}\pi^{2s+w-1}\frac{\Gamma(1-s-w/2+|m|)}{\Gamma(s+w/2+|m|)}\\\times
\sum_{n\neq0}\left(\frac{n}{|n|}\right)^{2m}|n|^{2s+w-2}
\sum_{q\neq0}\frac{1}{|q|^{2+2s}}
\sum_{c\Mod{q}}S(c,c;q)e[n\bar{c/q}]=\\
(-1)^{|m|}\pi^{2s+w-1}\frac{\Gamma(1-s-w/2+|m|)}{\Gamma(s+w/2+|m|)}
\frac{\zeta_{\kk}(s)}{\zeta_{\kk}(2s)}\\\times
\sum_{n\neq0}\left(\frac{n}{|n|}\right)^{2m}|n|^{2s+w-2}
\mathscr{L}_{\kk}(s;\bar{n}^2-4).
\end{multline}
To evaluate the residue at $w=2-2s$, we use \eqref{Sigma5} and \eqref{eq:sumofklsums}, obtaining
\begin{multline}\label{Sigma6}
\Sigma(s)=
-\frac{8}{\pi^2}\frac{\zeta_{\kk}(s)}{\zeta_{\kk}(2s)}\mathscr{L}_{\kk}(s;-4)
\int_0^{\pi/2}
\frac{h^{*}(0,\tau;2-2s)}{\cos(\pi s)}(\cos\tau)^{-2+2s}\frac{\pi d\tau}{(2\pi)^{2-2s}}+\\
\frac{8\pi^{2s-1}}{\pi^2}
\int_0^{\pi/2}
\mathop{{\sum}^*}_{m=0}^{\infty}
\left(\left(\frac{n}{|n|}\right)^{2m}+\left(\frac{n}{|n|}\right)^{-2m}
\right)
\sum_{n\neq0}\frac{\zeta_{\kk}(s)}{\zeta_{\kk}(2s)}\mathscr{L}_{\kk}(s;\bar{n}^2-4)|n|^{2s-2}
\\\times
\frac{1}{2\pi i}\int_{(a_1)}\frac{h^{*}(m,\tau;w)}{\cos(\pi w/2)}(\cos\tau)^{-w}
\frac{\Gamma(1-s-w/2+m)}{\Gamma(s+w/2+m)}
\left(\frac{|n|}{2}\right)^w
dwd\tau.
\end{multline}
It follows from equations \eqref{I integral def} and \eqref{I0 integral def} that
\begin{multline}\label{Sigma7}
\Sigma(s)=
\frac{8(2\pi)^{2s-1}}{\pi^2}
\int_0^{\pi/2}
\frac{\zeta_{\kk}(s)}{\zeta_{\kk}(2s)}\mathscr{L}_{\kk}(s;-4)
I(0,\tau,s)d\tau+\\
\frac{8(2\pi)^{2s-1}}{\pi^2}
\int_0^{\pi/2}
\frac{\zeta_{\kk}(s)}{\zeta_{\kk}(2s)}\sum_{n\neq0}\mathscr{L}_{\kk}(s;\bar{n}^2-4)
I(n,\tau,s)d\tau.
\end{multline}
This formula proves the analytic continuation of $\Sigma(s)$ to the region $\Re(s)>1/2$ since $I(n,\tau,s)$ is analytic for $\Re(s)>1/2$ and is of rapid decay in the $n$ aspect. This can be easily shown by moving the line of integration in \eqref{I integral def} to the left.
\end{proof}

\subsection{Properties of the weight function}
Our next goal is to evaluate the function $I(n,\tau,s)$ defined by \eqref{I integral def}.


\begin{lem}\label{integral hstar-Gamma} For a positive integer $m$ and $\Re(s)>3/4$  we have
\begin{multline}\label{integral with hstar-Gamma}
\frac{1}{2\pi i}\int_{(a)}\frac{h^{*}(m,\tau;w)}{\cos(\pi w/2)}(\cos\tau)^{-w}
\frac{\Gamma(1-s-w/2+m)}{\Gamma(s+w/2+m)}x^wdw=\\
2\int_0^{\infty}\psi(m,\tau;z)J_{2m}(2xz)(xz)^{2-2s}\frac{dz}{z},
\end{multline}
where $-2m-2N-2<a<2-2\Re(s)+2m.$
\end{lem}
\begin{proof}
First, we assume that $\Re(s)>1$ and choose $a$ such that
\begin{equation}\label{a condition}
0>a>\max(2-2\Re(s),-2m-2N-2).
\end{equation}
Due to \eqref{Mellin def} and \eqref{Mellin for psi} we have
\begin{equation}\label{h star zero}
\frac{h^{*}(m,\tau;w)}{\cos(\pi w/2)}(\cos\tau)^{-w}=\hat{\psi}(m,\tau;w).
\end{equation}
Therefore,
\begin{multline}\label{integral with hstar-Gamma2}
\frac{1}{2\pi i}\int_{(a)}\frac{h^{*}(m,\tau;w)}{\cos(\pi w/2)}(z\cos\tau)^{-w}
\frac{\Gamma(1-s-w/2+m)}{\Gamma(s+w/2+m)}x^wdw=\\
\frac{1}{2\pi i}\int_{(a)}
\frac{\Gamma(1-s-w/2+m)}{\Gamma(s+w/2+m)}x^w
\left(\int_0^{\infty}\psi(m,\tau;z)z^{w-1}dz\right) dw.
\end{multline}
The conditions  \eqref{a condition} on $a$, the Stirling formula \eqref{Stirling2} and the estimate \eqref{psi estimate} guarantee the absolute convergence of the double integral on the right-hand side of \eqref{integral with hstar-Gamma2}. This fact allows us to change the order of integration in \eqref{integral with hstar-Gamma2}, getting
\begin{multline}\label{integral with hstar-Gamma3}
\frac{1}{2\pi i}\int_{(a)}\frac{h^{*}(m,\tau;w)}{\cos(\pi w/2)}(z\cos\tau)^{-w}
\frac{\Gamma(1-s-w/2+m)}{\Gamma(s+w/2+m)}x^wdw=\\
\int_0^{\infty}\psi(m,\tau;z)\frac{1}{2\pi i}\int_{(a)}
\frac{\Gamma(1-s-w/2+m)}{\Gamma(s+w/2+m)}(zx)^wdw\frac{dz}{z}.
\end{multline}
Since the inner integral over $w$ on the right-hand side of \eqref{integral with hstar-Gamma3} is equal to $2(xz)^{2-2s}J_{2m}(2xz)$, we obtain \eqref{integral with hstar-Gamma} for $\Re(s)>1$.
\par
It follows from \eqref{psi estimate} and \eqref{Stirling2} that the integral on the left-hand side of \eqref{integral with hstar-Gamma} is absolutely convergent for $\Re(s)>1/2.$ Furthermore, the estimate \eqref{psi estimate} and standard estimates for the Bessel function show that the integral on the right-hand side of \eqref{integral with hstar-Gamma} is absolutely convergent for $\Re(s)>3/4.$ Thus we proved \eqref{integral with hstar-Gamma} for $\Re(s)>3/4.$
\end{proof}

\begin{lem} Let $0\le\vartheta\le\pi/2$ and $a,b,z>0.$ We have
\begin{multline}\label{Bessel sum over m}
J_{0}(az)J_{0}(bz)+2\sum_{m=1}^{\infty}\cos(2m\vartheta)J_{2m}(az)J_{2m}(bz)=\\
\frac{1}{2}J_0\left(z\sqrt{a^2+b^2+2ab\cos\vartheta}\right)+
\frac{1}{2}J_0\left(z\sqrt{a^2+b^2-2ab\cos\vartheta}\right).
\end{multline}
\end{lem}
\begin{proof}
Applying \cite[5.7.17.8]{PBM}, we obtain
\begin{multline}\label{Bessel sum over m2}
2\sum_{m=1}^{\infty}\cos(2m\vartheta)J_{2m}(az)J_{2m}(bz)=\\
-J_{0}(az)J_{0}(bz)+\frac{1}{2}\sum_{k=k_1}^{k_2}\alpha_kJ_0\left(z\sqrt{a^2+b^2-2ab(-1)^k\cos\vartheta}\right),
\end{multline}
where $k_1=-[1+\vartheta/\pi],k_2=[1-\vartheta/\pi]$ and $\alpha_k=1$ for $k\neq k_{1,2}.$ If $\vartheta/\pi\in\Z$ then $\alpha_{k_{1,2}}=1/2$, otherwise $\alpha_{k_{1,2}}=1.$ Thus for $\vartheta=0$ we have $k_1=-1, k_2=1, \alpha_{\pm1}=1/2$, which proves \eqref{Bessel sum over m}. For $0<\vartheta\le\pi/2$ we have $k_1=-1, k_2=0, \alpha_{-1}= \alpha_{0}=1$, obtaining \eqref{Bessel sum over m} once again.
\end{proof}


\begin{lem} Let $\vartheta=\arg(n).$ For $1/2<\Re(s)<1$ we have
\begin{multline}\label{I integral}
I(n,\tau,s)=
\frac{1}{4}\int_{-\infty}^{\infty}r^2h(r)\cosh(\pi r)\sum_{\pm}
\int_0^{\infty}
J_0\left(z\sqrt{|n|^2+4\sin^2\tau\pm4|n|\sin\tau\cos\vartheta}\right)\\ \times
K_{2ir}(2z\cos\tau)z^{1-2s}
dzdr.
\end{multline}
\end{lem}
\begin{proof}
Applying \eqref{integral with hstar-Gamma}, we obtain for $\Re(s)>3/4$
\begin{multline}\label{I integral2}
I(n,\tau,s)=
\mathop{{\sum}^*}_{m=0}^{\infty}
\left(\left(\frac{n}{|n|}\right)^{2m}+\left(\frac{n}{|n|}\right)^{-2m}
\right)
\int_0^{\infty}\psi(m,\tau;z)J_{2m}(z|n|)z^{1-2s}dz.
\end{multline}
Using \eqref{psi def} and \eqref{g def} we have for $1/2<\Re(s)<1$
\begin{multline}\label{I integral3}
I(n,\tau,s)=
\frac{1}{2}\int_{-\infty}^{\infty}r^2h(r)\cosh(\pi r)
\int_0^{\infty}K_{2ir}(2z\cos\tau)\\ \times
\mathop{{\sum}^*}_{m=0}^{\infty}
\left(\left(\frac{n}{|n|}\right)^{2m}+\left(\frac{n}{|n|}\right)^{-2m}
\right)J_{2m}(2z\sin\tau)J_{2m}(z|n|)z^{1-2s}
dzdr.
\end{multline}
Applying \eqref{Bessel sum over m}, we obtain
\begin{multline}\label{sum over m}
\mathop{{\sum}^*}_{m=0}^{\infty}
\left(\left(\frac{n}{|n|}\right)^{2m}+\left(\frac{n}{|n|}\right)^{-2m}\right)
J_{2m}(z|n|)J_{2m}(2z\sin\tau)=\\
\frac{1}{2}\sum_{\pm}J_0\left(z\sqrt{|n|^2+4\sin^2\tau\pm4|n|\sin\tau\cos\vartheta}\right).
\end{multline}
Finally, substituting \eqref{sum over m} in \eqref{I integral3} yields \eqref{I integral}.
\end{proof}


Let $\vartheta=\arg(n)$ and
\begin{equation}\label{c def}
c_{\pm}=\left(1\pm\frac{4\sin\tau\cos\vartheta}{|n|}+\frac{4\sin^2\tau}{|n|^2}\right)^{1/2},
\end{equation}
\begin{equation}\label{x def}
x_{\pm}(n,\tau)=\frac{|n|^2+4\sin^2\tau\pm4|n|\sin\tau\cos\vartheta}{(2\cos\tau)^{2}}=\frac{(|n|c_{\pm})^2}{(2\cos\tau)^{2}}.
\end{equation}
In some cases for simplicity we will write $x_{\pm}$ instead of $x_{\pm}(n,\tau).$
\begin{lem}\label{lem i1}  For $\Re(s)<1$ we have
\begin{multline}\label{I integral hypergeom1}
I(n,\tau,s)=
\frac{1}{4}\int_{-\infty}^{\infty}r^2h(r)\cosh(\pi r)\sum_{\pm}
\frac{\Gamma(1-s-ir)\Gamma(1-s+ir)}{4(\cos\tau)^{2-2s}}\\ \times
F\left(1-s+ir,1-s-ir,1;-x_{\pm}(n,\tau)\right)dr.
\end{multline}
\end{lem}
\begin{proof}
For $1/2<\Re(s)<1$ the formula \eqref{I integral hypergeom1} follows from \eqref{I integral} and \cite[6.576.3]{GR}. The right-hand side of \eqref{I integral hypergeom1} yields the analytic continuation of $I(n,\tau,s)$ to the region $\Re(s)<1.$
\end{proof}


\begin{lem}\label{lem i2} For $\Re(s)<1$ we have
\begin{multline}\label{I integral hypergeom2}
I(n,\tau,s)=\sum_{\pm}\frac{1}{8(x_{\pm}(n,\tau)\cos^2\tau)^{1-s}}
\int_{-\infty}^{\infty}r^2h(r)\cosh(\pi r)x_{\pm}^{-ir}\\
\frac{\Gamma(1-s+ir)\Gamma(-2ir)}{\Gamma(s-ir)}
F\left(1-s+ir,1-s+ir,1+2ir;\frac{-1}{x_{\pm}(n,\tau)}\right)dr.
\end{multline}
\end{lem}
\begin{proof}
The formula \eqref{I integral hypergeom2} follows from \eqref{I integral hypergeom1}  and \cite[p.117, (34)]{BE}. Indeed,
\begin{multline*}\label{hypergeometric transform}
\frac{\Gamma(1-s-ir)\Gamma(1-s+ir)}{4(\cos\tau)^{2-2s}}
F\left(1-s+ir,1-s-ir,1;-\frac{(|n|c_{\pm})^2}{(2\cos\tau)^{2}}\right)=\\
\left(\frac{|n|c_{\pm}}{2\cos\tau}\right)^{-2(1-s+ir)}
\frac{\Gamma(-2ir)\Gamma(1-s+ir)}{4(\cos\tau)^{2-2s}\Gamma(s-ir)}\\\times
F\left(1-s+ir,1-s+ir,1+2ir;-\frac{(2\cos\tau)^{2}}{(|n|c_{\pm})^2}\right)+\\
\left(\frac{|n|c_{\pm}}{2\cos\tau}\right)^{-2(1-s-ir)}
\frac{\Gamma(2ir)\Gamma(1-s-ir)}{4(\cos\tau)^{2-2s}\Gamma(s+ir)}\\\times
F\left(1-s-ir,1-s-ir,1-2ir;-\frac{(2\cos\tau)^{2}}{(|n|c_{\pm})^2}\right).
\end{multline*}
\end{proof}

\begin{lem}\label{lemma I estimate 1} For $0<\Re(s)<1$ we have
\begin{equation}\label{I integral estimate1}
I(n,\tau,s)\ll|n|^{-N-1/2},
\end{equation}
where $N$ is such that \eqref{conditions on h 1} is satisfied.
\end{lem}
\begin{proof}
Applying \cite[15.6.1]{HMF}, we obtain for $0<\Re(s)<1$
\begin{multline}\label{Euler intagral}
F\left(1-s+ir,1-s+ir,1+2ir;\frac{-1}{x_{\pm}}\right)=\\
\frac{\Gamma(1+2ir)}{\Gamma(s+ir)\Gamma(1-s+ir)}
\int_0^1
\frac{y^{-s}(1-y)^{s-1}}{(1+y/x_{\pm})^{1-s}}\left(\frac{y(1-y)}{1+y/x_{\pm}}\right)^{ir}dy.
\end{multline}
Thus substituting \eqref{Euler intagral} to \eqref{I integral hypergeom2} gives
\begin{multline}\label{I integral hypergeom3}
I(n,\tau,s)=\sum_{\pm}\frac{1}{8(x_{\pm}\cos^2\tau)^{1-s}}
\int_0^1
\frac{y^{-s}(1-y)^{s-1}}{(1+y/x_{\pm})^{1-s}}\\\times
\int_{-\infty}^{\infty}r^2h(r)\cosh(\pi r)
\frac{\Gamma(1+2ir)\Gamma(-2ir)}{\Gamma(s+ir)\Gamma(s-ir)}
\left(\frac{y(1-y)}{y+x_{\pm}}\right)^{ir}drdy.
\end{multline}
Applying \cite[5.5.3]{HMF}, we obtain
\begin{multline}\label{I integral hypergeom4}
I(n,\tau,s)=\sum_{\pm}\frac{\pi i}{16(x_{\pm}\cos^2\tau)^{1-s}}
\int_0^1
\frac{y^{-s}(1-y)^{s-1}}{(1+y/x_{\pm})^{1-s}}\\\times
\int_{-\infty}^{\infty}\frac{r^2h(r)}{\sinh(\pi r)\Gamma(s+ir)\Gamma(s-ir)}
\left(\frac{y(1-y)}{y+x_{\pm}}\right)^{ir}drdy.
\end{multline}
Using the property \eqref{conditions on h 1} we move the line of integration to $\Im(r)=-N-1/2$. The statement follows.
\end{proof}


\subsection{Exact formula for $1/2\le\Re{s}<1$}

The main result of this section is the following theorem.
\begin{thm}\label{1moment exact formula} For $1/2\le\Re{s}<1$ we have
\begin{multline}\label{firstmoment=MT+Sigma-ContSpectr2}
M_1(s)=
\frac{4\zeta_{\kk}(s)}{\pi^2}\int_{-\infty}^{\infty}r^2h(r)dr+\Sigma(s)
-32\pi^2\frac{\zeta_{\kk}(s)\zeta_{\kk}(2s-1)}{\zeta_{\kk}(2s)\zeta_{\kk}(2-s)}h(i(s-1))\\
-8\pi\frac{\zeta^2_{\kk}(s)}{\zeta_{\kk}(2s)}\int_{-\infty}^{\infty}
\frac{\zeta_{\kk}(s+ir)\zeta_{\kk}(s-ir)}{\zeta_{\kk}(1+ir)\zeta_{\kk}(1-ir)}h(r)dr,
\end{multline}
where $\Sigma(s)$ is given by \eqref{Sigma1} with $I(n, \tau, s)$ as in Lemmas \ref{lem i1} and \ref{lem i2}.
\end{thm}
\begin{proof}
The case $1/2<\Re{s}<1$ follows mainly from \eqref{firstmoment=MT+Sigma-ContSpectr} and Lemma \ref{Sigma Lemma}, where the analytic continuation of $\Sigma(s)$ is proven. The last thing to show is that the  integral on the right-hand side of \eqref{firstmoment=MT+Sigma-ContSpectr} can be continued analytically to the region of interest. This can be done in the same way as in \cite[Theorem 7.3]{BF3}. Furthermore, Lemma \ref{lemma I estimate 1} guarantees the absolute convergence of the sums over $n$ in $\Sigma(s)$ for $0<\Re{s}<1.$
\par
To prove the analytic continuation of $M_1(s)$ to $\Re(s)=1/2$, and specially to the point $s=1/2$, it is required to  examine the integral
\begin{equation}\label{I integral1}
\int_0^{\pi/2} I(n,\tau,s)d\tau
\end{equation}
more carefully. According to \eqref{I integral hypergeom1}, the function $I(n,\tau,s)$ has a multiple $(\cos\tau)^{2-2s}$, which can have a pole at the point $s=1/2.$ But note that as $\tau\rightarrow\pi/2$, the argument of the hypergeometric function in
\eqref{I integral hypergeom1}, in most cases,  tends to $-\infty.$ In such cases it is more reasonable to use the expression \eqref{I integral hypergeom2} for the function $I(n,\tau,s).$ Indeed, if $(|n|c_{\pm})^2\neq0$ for $\tau=\pi/2$ the right-hand side of
\eqref{I integral hypergeom2} does not cause any problem. So we are left to analyse for which $n$ one has $(|n|c_{\pm})^2=0$ with $\tau=\pi/2.$ Due to \eqref{c def} we have
\begin{multline}\label{small n}
(|n|c_{\pm})^2=|n|^2+4\sin^2\tau\pm4|n|\sin\tau\cos\vartheta\ge\\
|n|^2+4\sin^2\tau-4|n|\sin\tau
=(|n|-2\sin\tau)^2,
\end{multline}
and the equality may be reached only for $n\in\Z$ since it is required that $\cos\vartheta=\pm1.$ As we see $(|n|c_{\pm})^2=0$ only if $n=\pm1$ or $n=\pm2.$ In the first case: $n=\pm1$, the root is at the point $\tau=\pi/6$, and this does not cause any problem in \eqref{I integral1}. In the second case: $n=\pm2$,  we obtain
\begin{multline}\label{small n2}
\frac{(|n|c_{\pm})^2}{(2\cos\tau)^{2}}=\frac{(1-\sin\tau)^2}{(\cos\tau)^{2}}=\tan^2\left(\frac{\pi}{4}-\frac{\tau}{2}\right)=
\frac{1}{4}\left(\frac{\pi}{2}-\tau\right)^2+O\left(\left(\frac{\pi}{2}-\tau\right)^4\right).
\end{multline}
Thus for $n=\pm2$ the leading term of $I(n,\tau,s)$ as $\tau\rightarrow\pi/2$ is given by (see \eqref{I integral hypergeom1})
\begin{multline}\label{I integral n=2}
I(\pm2,\tau,s)=
\frac{1}{4}\int_{-\infty}^{\infty}r^2h(r)\cosh(\pi r)
\frac{\Gamma(1-s-ir)\Gamma(1-s+ir)}{4(\cos\tau)^{2-2s}}dr\\
+O\left(\left(\frac{\pi}{2}-\tau\right)^2\right).
\end{multline}
Therefore,  using \cite[5.12.2]{HMF} we have
\begin{multline}\label{I integral n=2 integral tau}
\int_0^{\pi/2} I(\pm2,\tau,s)d\tau=
\frac{\Gamma(1/2)\Gamma(s-1/2)}{32\Gamma(s)}\\\times
\int_{-\infty}^{\infty}r^2h(r)\cosh(\pi r)
\Gamma(1-s-ir)\Gamma(1-s+ir)dr+
O(1).
\end{multline}
According to \eqref{L beautiful 0}, the contribution of the summands with $n=\pm2$ in \eqref{Sigma1} to $\Sigma(s)$ is given by
\begin{multline}\label{Sigma n=2}
\frac{8(2\pi)^{2s-1}}{\pi^2}
\frac{\zeta_{\kk}(s)\zeta_{\kk}(2s-1)}{\zeta_{\kk}(2s)}
\frac{\Gamma(1/2)\Gamma(s-1/2)}{4\Gamma(s)}\\\times
\int_{-\infty}^{\infty}r^2h(r)\cosh(\pi r)
\Gamma(1-s-ir)\Gamma(1-s+ir)dr+
O(1).
\end{multline}
Since $\zeta_{\kk}(2s)$ has a pole at $s=1/2$, the whole expression \eqref{Sigma n=2}, as well as $\Sigma(s)$, are holomorphic at the point $s=1/2.$
\end{proof}

\begin{rem}
If we consider the first moment of symmetric square $L$-functions instead of Rankin-Selberg $L$-functions, among some other minor changes, it is required to multiply  the right-hand side of the equation \eqref{firstmoment=MT+Sigma-ContSpectr2} by $\zeta_{\kk}(2s).$ In doing so, the factor $\zeta_{\kk}(2s)$ disappears from \eqref{Sigma n=2}. Nevertheless, the analogue of
\eqref{firstmoment=MT+Sigma-ContSpectr2} is still true at the point $s=1/2$ because the poles from two summands cancel each other. This means that the expression
\begin{multline}\label{poles cancellation1}
\frac{4\zeta_{\kk}(s)\zeta_{\kk}(2s)}{\pi^2}\int_{-\infty}^{\infty}r^2h(r)dr+
\frac{8(2\pi)^{2s-1}}{\pi^2}
\frac{\zeta_{\kk}(s)\zeta_{\kk}(2s-1)\Gamma(1/2)\Gamma(s-1/2)}{4\Gamma(s)}\\\times
\int_{-\infty}^{\infty}r^2h(r)\cosh(\pi r)
\Gamma(1-s-ir)\Gamma(1-s+ir)dr
\end{multline}
is holomorphic at  $s=1/2.$ To prove this, let $s=1/2+u$ and apply the functional equation
\begin{equation}\label{dedeking func.eq}
\zeta_{\kk}(2u)=\pi^{4u-1}\zeta_{\kk}(1-2u)\frac{\Gamma(1-2u)}{\Gamma(2u)}.
\end{equation}
Thus \eqref{poles cancellation1} transforms into
\begin{multline}\label{poles cancellation2}
\zeta_{\kk}(1+2u)\frac{4\zeta_{\kk}(1/2+u)}{\pi^2}\int_{-\infty}^{\infty}r^2h(r)dr+\\
\zeta_{\kk}(1-2u)\frac{2(2\pi^3)^{2u}}{\pi^3}
\frac{\zeta_{\kk}(1/2+u)\Gamma(1/2)}{\Gamma(1/2+u)\Gamma(1-2u)}\frac{\Gamma(u)}{\Gamma(2u)}\\\times
\int_{-\infty}^{\infty}r^2h(r)\cosh(\pi r)
\Gamma(1/2-u-ir)\Gamma(1/2-u+ir)dr.
\end{multline}
To show that the expression \eqref{poles cancellation2} is holomorphic at  the point $u=0$, it is enough to prove that the coefficient before $u^{-1}$ in the Laurent expansion is zero. Hence it is required that
\begin{multline}\label{poles cancellation3}
\frac{4\zeta_{\kk}(1/2)}{\pi^2}\int_{-\infty}^{\infty}r^2h(r)dr-
\frac{4\zeta_{\kk}(1/2)}{\pi^3}\\\times
\int_{-\infty}^{\infty}r^2h(r)\cosh(\pi r)
\Gamma(1/2-ir)\Gamma(1/2+ir)dr=0,
\end{multline}
and this follows immediately from \cite[5.4.4]{HMF}.
\end{rem}


\section{Proof of Theorems \ref{thm:spec.exp.sum new bound} and \ref{thmPrimeGeodesic new bound}}
Following the paper of Ivic and Jutila \cite{IvJut}, let us define
\begin{equation}\label{omega def}
\omega_T(r)=\frac{1}{G\pi^{1/2}}\int_T^{2T}\exp\left(-\frac{(r-K)^2}{G^2}\right)dK.
\end{equation}
For an arbitrary $A>1$ and some $c>0$ we have (see \cite{IvJut})
\begin{equation}\label{omega1}
\omega_T(r)=1+O(r^{-A})\text{ if } T+cG\sqrt{\log T}<r<2T-cG\sqrt{\log T},
\end{equation}
\begin{equation}\label{omega2}
\omega_T(r)=O((|r|+T)^{-A})\text{ if }
r<T-cG\sqrt{\log T}\text{ or } r>2T+cG\sqrt{\log T},
\end{equation}
and otherwise
\begin{equation}\label{omega3}
\omega_T(r)=1+O(G^3(G+\min(|r-T|,|r-2T|))^{-3}).
\end{equation}
Using the Weyl law \cite[Section 5.5]{EGM}, we obtain
\begin{equation}
S(T,X)=\sum_{0<i\ll\log T}\sum_{r_j}X^{ir_j}\omega_{T_i}(r_j)+O(T^{2+\epsilon}G),
\end{equation}
where $T_i=T2^{-i}.$ Following the approach of Koyama (see \cite{Koyama}, \cite{ChatCherLaak}) and applying Theorem \ref{thm:momentoscill}, we prove Theorem \ref{thm:spec.exp.sum new bound} and derive as an immediate consequence
Theorem \ref{thmPrimeGeodesic new bound}.

\section{Proof of Theorem \ref{thm:momentoscill}}
For  an arbitrary large integer $N$ we define
\begin{equation}\label{qN def}
q_N(r)=\frac{(r^2+1/4)\ldots(r^2+(N-1/2)^2)}{(r^2+100N^2)^N}
\frac{(r^2+1)\ldots(r^2+N^2)}{(r^2+100N^2)^N},
\end{equation}
\begin{multline}\label{hN def}
h(K,N,T,X;r)=X^{ir}q_N(r)\exp\left(-\frac{(r-K)^2}{G^2}\right)+
X^{-ir}q_N(r)\exp\left(-\frac{(r+K)^2}{G^2}\right),
\end{multline}
where $T^{\epsilon}\ll G\ll T^{\theta/2+\epsilon}.$

 For simplicity, we will write  $h(*;r_j)$ instead of $h(K,N,T,X;r).$
\begin{lem}
Let $X\gg1$ and $X^{\epsilon}\le T\le X^{1/2+\epsilon}.$ Then for $s=1/2+it$, $|t|\ll T^{\epsilon}$ we have
\begin{multline}\label{1moment smoth1}
\sum_{r_j}\frac{r_j}{\sinh(\pi r_j)}\omega_T(r_j)X^{ir_j}L(u_{j}\otimes u_{j},s)=\\
\frac{1}{G\pi^{1/2}}\int_T^{2T}
\sum_{r_j}\frac{r_j}{\sinh(\pi r_j)}h(*;r_j)L(u_{j}\otimes u_{j},s)dK+O(T^{7/6+\epsilon}).
\end{multline}
\end{lem}
\begin{proof}
This follows from the definitions of functions $\omega_T(r_j)$ and $h(*;r)$ (see \eqref{omega def}, \eqref{hN def}), the fact that $q_N(r)=1+O(r^{-2})$ and \cite[Corollary 3.4]{ChatCherLaak}.
\end{proof}
For the sum on the right-hand side of \eqref{1moment smoth1} we apply Theorem \ref{1moment exact formula}.  For simplicity, we consider only $s=1/2+it$ with $1\le|t|\ll T^{\epsilon}.$ The cases $|t|\le1$ and $t=0$ require only minor changes.
\par
To evaluate the integrals we will frequently use equations \cite[3.323.2]{GR} and \cite[3.462.2]{GR}, namely
\begin{equation}\label{exp integral1}
\int_{-\infty}^{\infty}\exp(-p^2x^2+qx)dx=\frac{\pi^{1/2}}{p}\exp\left(\frac{q^2}{4p^2}\right),
\quad\hbox{if}\quad\Re(p^2)>0,
\end{equation}
\begin{equation}\label{exp integral2}
\int_{-\infty}^{\infty}x^n\exp(-x^2+qx)dx=P_n(q)\exp\left(\frac{q^2}{4}\right),
\end{equation}
where $n$ is a nonnegative integer and $P_n(q)$ is a polynomial of degree $n.$ Another important ingredient of our proof is the following lemma.
\begin{lem}\label{lemma on small I}
Suppose that $X\gg1$ and $X^{\epsilon}\le T\le X^{1/2+\epsilon}.$ Let $s=1/2+it$, $1\ll|t|\ll T^{\epsilon}$ and $n\neq0,\pm1,\pm2.$ If
\begin{equation}\label{conditions for Ismall}
x_{\pm}\gg X^{1+\epsilon}\quad\hbox{or}\quad X^{-1+\epsilon}\ll x_{\pm}\ll X^{1-\epsilon},
\end{equation}
where $x_{\pm}$ is defined by \eqref{x def},
then  for any fixed $A>1$ we have
\begin{equation}\label{I integral small}
\frac{1}{G}\int_T^{2T}I(n,\tau,s)dK\ll X^{-A}.
\end{equation}
\end{lem}
\begin{proof}
According to \eqref{small n}, the condition $n\neq0,\pm1,\pm2$ guarantees that 
\begin{equation*}
4x_{\pm}\cos^2\tau=(|n|c_{\pm})^2\neq0.
\end{equation*}
Consider the case $x_{\pm}\gg X^{1+\epsilon}.$ Using \eqref{I integral hypergeom4}, we obtain
\begin{multline}\label{I integral estimate2}
\frac{1}{G}\int_T^{2T}I(n,\tau,s)dK\ll
\frac{1}{G}\int_T^{2T}\int_0^1
\frac{y^{-s}(1-y)^{s-1}}{(1+y/x_{\pm})^{1-s}}\\\times
\int_{-\infty}^{\infty}\frac{r^2h(*;r)}{\sinh(\pi r)\Gamma(s+ir)\Gamma(s-ir)}
\left(\frac{y(1-y)}{y+x_{\pm}}\right)^{ir}drdydK.
\end{multline}
The function $h(*;r)$ defined by \eqref{hN def} consists of two summands. We study only the summand with the multiple $X^{ir}$. The second summand can be treated similarly. 

Moving the line of integration in the integral over $r$ to $\Im(r)=-N-1/2$, we do not cross any pole due to the presence of the function $q_N(r)$ defined in \eqref{qN def}. Since $|t|\ll T^{\epsilon}$ and $r\asymp K\asymp T$, the Stirling formula \eqref{Stirling2} implies that $|\Gamma(s+ir)\Gamma(s-ir)|\asymp\exp(-\pi|r|).$ Thus we obtain
\begin{equation}\label{I integral estimate2.1}
\frac{1}{G}\int_T^{2T}I(n,\tau,s)dK\ll
\left(\frac{X}{x_{\pm}}\right)^{N+1/2}.
\end{equation}
Since  $N$ can be chosen to be an arbitrary constant, we prove \eqref{I integral small} for  $x_{\pm}\gg X^{1+\epsilon}.$  Note that in the same way we can prove \eqref{I integral small} for the summand of $h(*;r)$ with the multiple $X^{-ir}$ when  $x_{\pm}\gg X^{-1+\epsilon}.$
\par
Consider the case $X^{-1+\epsilon}\ll x_{\pm}\ll X^{1-\epsilon}$ and the summand of $h(*;r)$ with the multiple $X^{ir}$. In that case we start by evaluating the integral over $r$ in \eqref{I integral estimate2}. The strategy is as follows.  First, we expand the function under the integration sign in Taylor series at the point $r=K$. Second, we make the change of variable $r=K+Gu$ and apply \eqref{exp integral2}.  By taking sufficiently many terms in the Taylor expansion, we obtain a  negligibly small error.  All such terms can be treated similarly. Thus we consider further only the main term as it gives the largest contribution. By the Stirling formula \eqref{Stirling2}
\begin{multline}\label{Taylor1}
\sinh(\pi r)\Gamma(s+ir)\Gamma(s-ir)=\pi\exp(\pi(|r|-|r+t|/2-|r-t|/2))\\\times
\exp(i( (r+t)\log|r+t|-(r-t)\log|r-t|-2t ) )(1+\ldots).
\end{multline}
Having in mind that $r\asymp K\asymp T$ and $|t|\ll T^{\epsilon}$, we obtain
\begin{multline}\label{Taylor2}
(r+t)\log|r+t|-(r-t)\log|r-t|-2t=2t\log r\\
+r\log(1+t/r)-r\log(1-t/r)-2t+t\log(1-t^2/r^2)=\\
2t\log K+2t\log(1+(r-K)/K)-\frac{t^3}{3r^2}\ldots
\end{multline}
Since we can restrict the $r$ integral to the region $|r-K|\ll GT^{\epsilon}$, it is possible to expand $\log(1+(r-K)/K)$ in the Taylor series with a negligibly small error. Substituting \eqref{Taylor2} to \eqref{Taylor1}, we obtain
\begin{equation}\label{Taylor3}
\sinh(\pi r)\Gamma(s+ir)\Gamma(s-ir)=\pi\exp(2it\log K)(1+\ldots).
\end{equation}
It follows from \eqref{Taylor3} and \eqref{hN def}  that the main term of the the integral over $r$ in \eqref{I integral estimate2} is equal to
\begin{equation}\label{I integral estimate3}
\pi K^{2it}\int_{-\infty}^{\infty}r^2\exp\left(-\frac{(r-K)^2}{G^2}\right)\left(\frac{Xy(1-y)}{y+x_{\pm}}\right)^{ir}dr.
\end{equation}
Making the change of variable $r=K+Gu$ and applying \eqref{exp integral1}, we obtain that the main term is
\begin{equation}\label{I integral estimate4}
\pi^{3/2}K^{2it}\left(\frac{Xy(1-y)}{y+x_{\pm}}\right)^{iK}GK^2
\exp\left(-\frac{G^2}{4}\log^2\left(\frac{Xy(1-y)}{y+x_{\pm}}\right)\right).
\end{equation}
This expression is exponentially small unless
\begin{equation}\label{short interval1}
\left|\frac{Xy(1-y)}{y+x_{\pm}}-1\right|\ll\frac{\log X}{G}.
\end{equation}
The last condition is satisfied if either
\begin{equation}\label{short interval2}
\left|y-a\right|\ll a\frac{\log X}{G}\quad\hbox{or}\quad
\left|(1-y)-b\right|\ll b\frac{\log X}{G},
\end{equation}
where $a=x_{\pm}/X$ and $b=(1+x_{\pm})/X.$  

Let $\chi_a(y)$ and $\chi_b(y)$ be smooth characteristic functions of the intervals \eqref{short interval2} such that
\begin{equation}\label{j norms of chi a,b}
\chi_a^{(j)}\ll \left(a\frac{\log X}{G}\right)^{-j},\qquad
\chi_b^{(j)}\ll \left(b\frac{\log X}{G}\right)^{-j}.
\end{equation}
Thus we are left to estimate \eqref{I integral estimate2} with the additional multiples $\chi_{a}(y)$ and $\chi_{a}(y)$. To this end, we change the order of integration in \eqref{I integral estimate2}, making the integral over $y$ to be the inner one. First, consider the $"a"$ case
\begin{multline}\label{I integral estimate5}
\frac{x_{\pm}^{1-s}}{G}\int_T^{2T}
\int_{-\infty}^{\infty}\frac{r^2h(*;r)}{\sinh(\pi r)\Gamma(s+ir)\Gamma(s-ir)}\\\times
\int_0^1\frac{y^{-s}(1-y)^{s-1}}{(y+x_{\pm})^{1-s}}
\left(\frac{y(1-y)}{y+x_{\pm}}\right)^{ir}\chi_{a,b}(y)dydrdK.
\end{multline}
We integrate $j$ times by parts using the factor $y^{-s+ir}$
\begin{multline}\label{I integral estimate6}
\int_0^1\frac{(1-y)^{s-1+ir}}{(y+x_{\pm})^{1-s+ir}}\chi_{a}(y)y^{-s+ir}dy\ll\\
\int_0^1
\left|\frac{d^{j}}{dy^j}\left(\frac{(1-y)^{s-1+ir}}{(y+x_{\pm})^{1-s+ir}}\chi_{a}(y)\right)\right|
\frac{y^{j-1/2}}{(|t|+|r|)^j}dy\ll\\
\int_0^1\left(y^j+\frac{y^j}{(y+x_{pm})^j}+\frac{y^j}{(|t|+|r|)^j}\left|\chi_a^{(j)}\right|\right)\chi_{a}(y)dy.
\end{multline}
Applying \eqref{short interval2} and \eqref{j norms of chi a,b}, we obtain
\begin{multline}\label{I integral estimate7}
\int_0^1\frac{(1-y)^{s-1+ir}}{(y+x_{\pm})^{1-s+ir}}\chi_{a}(y)y^{-s+ir}dy\ll
\left(\frac{x_{\pm}}{X}\right)^j+\left(\frac{1}{X}\right)^j+\left(\frac{G}{(|t|+|r|)}\right)^j\ll X^{-A}.
\end{multline}
Indeed, $x_{\pm}\ll X^{1-\epsilon}$ and the integral over $r$ can be restricted to the interval $|r-K|\ll G\log^2X$ so that
\begin{equation}
G/|r|\ll G/K\ll G/T\ll X^{-\epsilon}
\end{equation}
since $T\gg X^{\epsilon}$. Substituting \eqref{I integral estimate7} to \eqref{I integral estimate5} and estimating trivially the remaining  integrals, we prove the lemma.
\par
The $"b"$ case can be treated similarly. The only difference is that  it is required to integrate $j$ times by parts using the multiple $(1-y)^{s-1+ir}.$
\end{proof}

\begin{lem}\label{lemma I main estimate}
Suppose that $X\gg1$ and $X^{\epsilon}\le T\le X^{1/2+\epsilon}.$ Let $s=1/2+it, 1\ll|t|\ll T^{\epsilon}$ and $n\neq0,\pm1,\pm2.$ Let $A$ be an arbitrary positive constant. Let
\begin{equation}\label{main conditions for I}
X^{1-\epsilon}\ll x_{\pm}\ll X^{1+\epsilon},
\end{equation}
 where $x_{\pm}$ is defined by \eqref{x def}.
The following holds
\begin{equation}\label{I integral main estimate1}
\frac{1}{G}\int_T^{2T}I(n,\tau,s)dK\ll \sum_{\pm}\frac{T^{5/2}}{(x_{\pm}(n,\tau)\cos^2\tau)^{1/2}}
\end{equation}
if $|X-4x_{\pm}|\ll XT^{-1+\epsilon}$ ,
\begin{equation}\label{I integral main estimate2}
\frac{1}{G}\int_T^{2T}I(n,\tau,s)dK\ll \sum_{\pm}\frac{T^{3/2}}{(x_{\pm}(n,\tau)\cos^2\tau)^{1/2}}
\frac{X}{|4x_{\pm}-X|}
\end{equation}
if $XT^{-1+\epsilon}\ll|X-4x_{\pm}|\ll XG^{-1}\log T$,
and
\begin{equation}\label{I integral main estimate3}
\frac{1}{G}\int_T^{2T}I(n,\tau,s)dK\ll  ((1+|n|)X)^{-A}
\end{equation}
if $|X-4x_{\pm}|\gg (X\log T)/G$.
\end{lem}
\begin{proof}
The condition $n\neq0,\pm1,\pm2$ guarantees that (see \eqref{small n})
\begin{equation*}
4x_{\pm}\cos^2\tau=(|n|c_{\pm})^2\neq0.
\end{equation*}
The key idea is to apply the representation \eqref{I integral hypergeom2} for $I(n,\tau,s)$ and to prove an asymptotic formula for the hypergeometric function from \eqref{I integral hypergeom2}. 
With this goal, we use a technique based on the Mellin-Barnes integral representation for the hypergeometric function, which is possible since $r/x_{\pm}\approx T/X\ll X^{-1/2}$. 
\par
Writing the Mellin-Barnes integral \cite[15.6.1]{HMF}, we obtain
\begin{multline}
F\left(1-s+ir,1-s+ir,1+2ir;\frac{-1}{x_{\pm}}\right)=\\
\frac{\Gamma(1+2ir)}{\Gamma^2(1-s+ir)}\frac{1}{2\pi i}\int_{(-1/4)}
\frac{\Gamma^2(1-s+ir+z)\Gamma(-z)}{\Gamma(1+2ir+z)}x_{\pm}^{-z}dz.
\end{multline}
Moving the line of integration to the right on the line $\Re(z)=a+1/2$ with $a$ being a positive integer, we have
\begin{multline}\label{2F1 MellinBarnes1}
F\left(1-s+ir,1-s+ir,1+2ir;\frac{-1}{x_{\pm}}\right)=\\
\frac{\Gamma(1+2ir)}{\Gamma^2(1-s+ir)}\sum_{j=0}^{a}\frac{(-1)^j}{j!}
\frac{\Gamma^2(1-s+ir+j)}{\Gamma(1+2ir+j)}x_{\pm}^{-j}+\\
\frac{\Gamma(1+2ir)}{\Gamma^2(1-s+ir)}\frac{1}{2\pi i}\int_{-\infty}^{\infty}
\frac{\Gamma^2(a+1+i(r-t+u))\Gamma(-a-1/2-iu)}{\Gamma(a+3/2+i(2r+u))x_{\pm}^{+a+1/2+iu}}du.
\end{multline}
The contribution of all residue terms from \eqref{2F1 MellinBarnes1} to \eqref{I integral hypergeom2} can be treated similarly, and since $r/x_{\pm}\ll X^{-1/2}$ the main contribution comes from the summand with $j=0.$ To estimate the integral in
\eqref{2F1 MellinBarnes1}  we use the Stirling formula \eqref{Stirling2}. It shows that the interval $-r+t+1<u<-1$ gives the main contribution to the integral and is bounded by
\begin{multline}\label{2F1 MellinBarnes2}
\frac{|\Gamma(1+2ir)|}{|\Gamma^2(1-s+ir)|x_{\pm}^{a+1/2}}\exp(\pi t)\int_{-r+t+1}^{-1}
\frac{|r-t+u|^{2a+1}|u|^{-a-1}}{|2r+u|^{a+1}}du\ll\\
\frac{|\Gamma(1+2ir)|}{|\Gamma^2(1-s+ir)|x_{\pm}^{a+1/2}}\exp(\pi t)r^a\ll
\frac{r^{a+1/2}}{x_{\pm}^{a+1/2}}.
\end{multline}
 Choosing $a$ to be a sufficiently large integer, we obtain that this term is negligible  since $r/x_{\pm}\ll X^{-1/2}$. Consequently, in order to estimate $I(n,\tau,s)$ it is sufficient to consider \eqref{I integral hypergeom2} with  the hypergeometric function being replaced by one. Hence we need to estimate
\begin{multline}\label{I integral estimate8}
\sum_{\pm}\frac{1}{8(x_{\pm}(n,\tau)\cos^2\tau)^{1-s}}
\int_{-\infty}^{\infty}r^2h(*;r)\cosh(\pi r)x_{\pm}^{-ir}
\frac{\Gamma(1-s+ir)\Gamma(-2ir)}{\Gamma(s-ir)}dr.
\end{multline}
Using the Stirling formula \eqref{Stirling2} and arguing in the same way as in \eqref{Taylor1}-\eqref{Taylor3}, we have
\begin{equation}\label{Taylor4}
\cosh(\pi r)\frac{\Gamma(1-s+ir)\Gamma(-2ir)}{\Gamma(s-ir)}=
\frac{\pi^{1/2}\exp(\pi i/4)}{2r^{1/2}}2^{-2ir}K^{-2it}(1+\ldots).
\end{equation}
Substituting \eqref{Taylor4} to \eqref{I integral estimate8}, we obtain the main term
\begin{equation}\label{I integral estimate9}
\sum_{\pm}\frac{\pi^{1/2}\exp(\pi i/4)K^{-2it}}{16(x_{\pm}(n,\tau)\cos^2\tau)^{1-s}}
\int_{-\infty}^{\infty}r^{3/2}h(*;r)(4x_{\pm})^{-ir}dr.
\end{equation}
Using \eqref{hN def}, making the change of variable $r=K+Gu$  and applying \eqref{exp integral1}-\eqref{exp integral2}, the main term can be transformed into
\begin{equation}\label{I integral estimate10}
\sum_{\pm}\frac{\pi^{1/2}\exp(\pi i/4)K^{-2it}}{16(x_{\pm}(n,\tau)\cos^2\tau)^{1-s}}
K^{3/2}G \left(\frac{X}{4x_{\pm}}\right)^{iK}
\pi^{1/2}\exp\left(-\frac{G^2}{4}\log^2\frac{X}{4x_{\pm}}\right).
\end{equation}
It follows from \eqref{I integral estimate10} that
\begin{multline}\label{I integral main estimate4}
\frac{1}{G}\int_T^{2T}I(n,\tau,s)dK\ll
\sum_{\pm}\exp\left(-\frac{G^2}{4}\log^2\frac{X}{4x_{\pm}}\right)
\frac{1}{(x_{\pm}(n,\tau)\cos^2\tau)^{1/2}}\\ \times
\left|
\int_T^{2T}K^{3/2-2it}\left(\frac{X}{4x_{\pm}}\right)^{iK}dK
\right|.
\end{multline}
To prove \eqref{I integral main estimate3} for $|\log\frac{X}{4x_{\pm}}|\gg (\log T)/G,$ it is enough to estimate the integral trivially by $T^{5/2}$ since the exponential factor is negligibly small in this case.
\par
Now let us consider the case $|\log\frac{X}{4x_{\pm}}|\ll (\log T)/G.$ Once again we can estimate the integral in \eqref{I integral main estimate4} trivially by $T^{5/2}$ and obtain \eqref{I integral main estimate1}. To prove \eqref{I integral main estimate2} we  apply the first derivative test for the integral in \eqref{I integral main estimate4}. Let $L=T\log(X/(4x_{\pm}))$. We have
\begin{multline}\label{I integral main estimate5}
\left|
\int_T^{2T}K^{3/2-2it}\left(\frac{X}{4x_{\pm}}\right)^{iK}dK
\right|\ll
T^{5/2}\left|
\int_1^{2}y^{3/2-2it}\left(\frac{X}{4x_{\pm}}\right)^{iTy}dy
\right| \\ \ll
T^{5/2}\left|
\int_1^{2}y^{3/2}\exp(i(-2t\log y+Ly) )dy
\right|.
\end{multline}
The conditions in \eqref{I integral main estimate2}  ensure that $|L|\gg T^{\epsilon}$, and therefore, $|L|\gg |t|.$ Thus applying the first derivative test, we obtain
\begin{multline}\label{I integral main estimate6}
\left|
\int_T^{2T}K^{3/2-2it}\left(\frac{X}{4x_{\pm}}\right)^{iK}dK
\right|\ll
\frac{T^{5/2}}{|L|}=
\frac{T^{3/2}}{|\log(X/(4x_{\pm}))|}\ll
\frac{T^{3/2}X}{|X-4x_{\pm}|}.
\end{multline}
Substituting \eqref{I integral main estimate6} to \eqref{I integral main estimate4}, we prove \eqref{I integral main estimate2}.
\end{proof}

The main contribution to the right-hand side of \eqref{1moment smoth1} comes from the sum over large $|n|$ in  $\Sigma(s)$  and is estimated in the following lemma.

\begin{lem}
Let $X\gg1$ and $X^{\epsilon}\le T\le X^{1/2+\epsilon}.$ Then for $s=1/2+it$, $|t|\ll T^{\epsilon}$ we have
\begin{multline}\label{sum over n estimate}
\frac{1}{G\pi^{1/2}}\int_T^{2T}
\frac{8(2\pi)^{2s-1}}{\pi^2}
\int_0^{\pi/2}
\frac{\zeta_{\kk}(s)}{\zeta_{\kk}(2s)}
\sum_{n\neq0,\pm1,\pm2}\mathscr{L}_{\kk}(s;\bar{n}^2-4)
I(n,\tau,s)d\tau dK\\ \ll T^{3/2}X^{1/2+\theta}(TX)^{\epsilon}.
\end{multline}
\end{lem}
\begin{proof}
Using standard estimates on $\zeta_{\kk}(s)$ and \eqref{subconvexity}, we obtain that it is required to prove the following
\begin{equation}\label{sum over n estimate2}
\int_0^{\pi/2}\sum_{n\neq0,\pm1,\pm2}|n|^{2\theta+\epsilon}
\left|\frac{1}{G}\int_T^{2T}I(n,\tau,s)dK\right|d\tau \ll T^{3/2}X^{1/2+\theta}(TX)^{\epsilon}.
\end{equation}
The result of Lemma \ref{lemma on small I} shows that everything is small unless $X^{1-\epsilon}\ll x_{\pm}\ll X^{1+\epsilon}.$ In the later case, we apply Lemma \ref{lemma I main estimate}. For $|X-4x_{\pm}|\ll XT^{-1+\epsilon}$ we
apply \eqref{I integral main estimate1} and obtain
\begin{equation}\label{sum over n estimate3}
\sum_{\pm}\int_0^{\pi/2}
\sum_{\substack{n\neq0,\pm1,\pm2\\ |X-4x_{\pm}|\ll XT^{-1+\epsilon}}}
\frac{T^{5/2}|n|^{2\theta+\epsilon}}{(x_{\pm}(n,\tau)\cos^2\tau)^{1/2}}d\tau.
\end{equation}
Using \eqref{x def} and approximating the sum over $n$ by a double integral (note that since $T\ll X^{1/2+\epsilon}$ the error term caused by such approximation is less than the main term), we have
\begin{equation}\label{sum over n estimate4}
\sum_{\pm}\int_0^{\pi/2}
\sum_{\substack{n\neq0,\pm1,\pm2\\ |X\cos^2\tau-|n|^2c^2_{\pm}|\ll XT^{-1+\epsilon}\cos^2\tau}}
\frac{T^{5/2}|n|^{2\theta+\epsilon}}{|n|c_{\pm}}d\tau\ll
T^{3/2}X^{1/2+\theta}(TX)^{\epsilon}.
\end{equation}
The case when $XT^{-1+\epsilon}\ll|X-4x_{\pm}|\ll XG^{-1}\log T$  in \eqref{sum over n estimate2} can be treated analogously.

\end{proof}
Finally, we are left to estimate the remaining summands of $\Sigma(s)$ in \eqref{Sigma1} as well as the other summands from the right-hand side of \eqref{firstmoment=MT+Sigma-ContSpectr2}. The main term
\begin{equation*}
\frac{4\zeta_{\kk}(s)}{\pi^2}\int_{-\infty}^{\infty}r^2h(r)dr
\end{equation*}
of \eqref{firstmoment=MT+Sigma-ContSpectr2}  is negligible due to \eqref{hN def} and \eqref{exp integral1}-\eqref{exp integral2}. The term containing $h(*;i(s-1))$ is also negligible. Using the subconvexity bound $$\zeta_{\kk}(1/2+ir)\ll(1+|r|)^{1/6+\epsilon},$$ which is due to Kaufman \cite{Kauf},
and estimating the remaining term trivially, we obtain
\begin{equation*}
-8\pi\frac{\zeta^2_{\kk}(s)}{\zeta_{\kk}(2s)}\int_{-\infty}^{\infty}
\frac{\zeta_{\kk}(s+ir)\zeta_{\kk}(s-ir)}{\zeta_{\kk}(1+ir)\zeta_{\kk}(1-ir)}h(r)dr\ll K^{1/3+\epsilon}G.
\end{equation*}
Substitution of this estimate to \eqref{1moment smoth1} gives the error term of size $T^{4/3+\epsilon}$, which is smaller than \eqref{sum over n estimate}.
\par
The next lemma provides estimates for the remaining summands of $\Sigma(s)$.
\begin{lem}
Let $X\gg1$ and $X^{\epsilon}\le T\le X^{1/2+\epsilon}.$ Then for $s=1/2+it$, $|t|\ll T^{\epsilon}$ we have
\begin{multline}\label{sum over exceptional n estimate}
\frac{1}{G\pi^{1/2}}\int_T^{2T}
\frac{8(2\pi)^{2s-1}}{\pi^2}
\int_0^{\pi/2}
\frac{\zeta_{\kk}(s)}{\zeta_{\kk}(2s)}
\sum_{n=0,\pm1,\pm2}\mathscr{L}_{\kk}(s;\bar{n}^2-4)
I(n,\tau,s)d\tau dK \\ \ll \frac{T^{3/2}X^{\epsilon}}{X^{1/2}}.
\end{multline}
\end{lem}
\begin{proof}
Using standard estimates on the Dedekind zeta function $\zeta_{\kk}(s)$, we obtain that it is required  to prove
\begin{equation}\label{sum over exceotional n estimate2}
\int_0^{\pi/2}\sum_{n=0,\pm1,\pm2}
\left|\frac{1}{G}\int_T^{2T}I(n,\tau,s)dK\right|d\tau \ll \frac{T^{3/2}X^{\epsilon}}{X^{1/2}}.
\end{equation}
Let us first consider the case $n=0.$ In that case it follows from \eqref{I0 integral def} and \eqref{h star def2} that
\begin{multline}\label{I0 integral def2}
I(0,\tau,s)=
\frac{1}{2^4(\cos\tau)^{2-2s}}
\int_{-\infty}^{\infty}r^2h(r)\cosh(\pi r)
\Gamma(1-s+ir)\Gamma(1-s-ir)\\ \times
 F(1-s+ir,1-s-ir,1;-\tan^2\tau)dr.
\end{multline}
Using the transformation formula \cite[p.117, (34)]{BE} 
\begin{multline}\label{I0 integral def3}
I(0,\tau,s)=
\frac{1}{2^3(\sin\tau)^{2-2s}}
\int_{-\infty}^{\infty}r^2h(r)\cosh(\pi r)
\frac{\Gamma(-2ir)\Gamma(1-s+ir)}{\Gamma(s-ir)}\\ \times
F(1-s+ir,1-s-ir,1+2ir;-\tan^{-2}\tau)(\tan\tau)^{-2ir}dr.
\end{multline}
The right-hand side of \eqref{I0 integral def3} equals up to a constant to the right-hand side of \eqref{I integral hypergeom2} if we let $x_{\pm}=\tan^2\tau.$ Therefore, we can apply the results of  Lemmas \ref{lemma on small I}
and \ref{lemma I main estimate}.  According to Lemma \ref{lemma on small I}, the contribution of $\tan^2\tau\gg X^{1+\epsilon}$ and $X^{-1+\epsilon}\ll \tan^2\tau\ll X^{1-\epsilon}$ is bounded by $X^{-A}$ for any $A>0$.

To estimate the contribution of $X^{1-\epsilon}\ll\tan^2\tau\ll X^{1+\epsilon}$ we apply Lemma \ref{lemma I main estimate}, and making the change of variables $y=\tan\tau$,  obtain
\begin{multline}\label{estimate of n=0 contribution}
\int_0^{\pi/2}
\left|\frac{1}{G}\int_T^{2T}I(0,\tau,s)dK\right|d\tau \ll
\int_0^{X^{-1/2+\epsilon}}
\left|\frac{1}{G}\int_T^{2T}I(0,\tau,s)dK\right|d\tau+\\
\int_{|X-4\tan^2\tau|\ll XT^{-1+\epsilon}}
\frac{T^{5/2}d\tau}{\sin\tau}+X^{-A}\\+
\int_{XT^{-1+\epsilon}\ll|X-4\tan^2\tau|\ll XG^{-1}\log T}
\frac{T^{3/2}Xd\tau}{|X-4\tan^2\tau|\sin\tau}\ll\\
\int_0^{X^{-1/2+\epsilon}}\left|\frac{1}{G}\int_T^{2T}I(0,\tau,s)dK\right|d\tau+
\int_{|X-4y^2|\ll XT^{-1+\epsilon}}
\frac{T^{5/2}dy}{X}+X^{-A}\\+
\int_{XT^{-1+\epsilon}\ll|X-4y^2|\ll XG^{-1}\log T}
\frac{T^{3/2}dy}{|X-4y^2|}\\\ll
\int_0^{X^{-1/2+\epsilon}}\left|\frac{1}{G}\int_T^{2T}I(0,\tau,s)dK\right|d\tau+
\frac{T^{3/2}X^{\epsilon}}{X^{1/2}}.
\end{multline}
To estimate the remaining integral we use \eqref{I0 integral def2}. For the hypergeometric function in
\eqref{I0 integral def2} we apply the Mellin-Barnes integral representation \cite[15.6.1]{HMF}, getting
\begin{multline}\label{2F1 MellinBarnes5}
\Gamma(1-s+ir)\Gamma(1-s-ir) F(1-s+ir,1-s-ir,1;-\tan^2\tau)=\\
\frac{1}{2\pi i}\int_{(-1/4)}
\frac{\Gamma(1-s+ir+z)\Gamma(1-s-ir+z)\Gamma(-z)}{\Gamma(1+z)}(\tan\tau)^{-2z}dz.
\end{multline}
Moving the line of integration to  $\Re(z)=a+1/2$ with $a$ being a positive integer, we obtain
\begin{multline}\label{2F1 MellinBarnes3}
\Gamma(1-s+ir)\Gamma(1-s-ir)F(1-s+ir,1-s-ir,1;-\tan^2\tau)=\\
\sum_{j=0}^{a}\frac{(-1)^j}{j!}
\frac{\Gamma(1-s+ir+j)\Gamma(1-s-ir+j)}{\Gamma(1+j)}(\tan\tau)^{-2j}+\\
\frac{1}{2\pi i}\int_{-\infty}^{\infty}
\frac{\Gamma(a+1+i(r-t+u))\Gamma(a+1+i(u-t-r))\Gamma(-a-1/2-iu)}{\Gamma(a+3/2+iu)(\tan\tau)^{2a+1+2iu}}du.
\end{multline}
Contribution of all residue terms from \eqref{2F1 MellinBarnes3} to \eqref{I0 integral def2} can be treated in the same way. Since $r/\tan\tau\ll X^{-1/2+\epsilon}$, the summand with $j=0$ is the largest. To estimate the integral in
\eqref{2F1 MellinBarnes3}  we use the Stirling formula \eqref{Stirling2}.
It shows that the interval $-r+t+1<u<r+t-1$ gives the largest contribution to the integral (we assume for simplicity that $r>0$ since the left-hand side of \eqref{2F1 MellinBarnes5} is an even function in $r$), which is bounded by
\begin{multline}\label{2F1 MellinBarnes4}
(\tan\tau)^{-2a-1}
\exp(-\pi r)\int_{-r+t+1}^{r+t-1}
\frac{|r-t+u|^{a+1/2}|r+t-u|^{a+1/2}}{|u|^{2a+2}}du\ll\\
\exp(\pi r)(r/\tan\tau)^{2a+1}.
\end{multline}
 By choosing sufficiently large integer $a$, we prove that this term is negligible because $r/\tan\tau\ll X^{-1/2+\epsilon}$.
Therefore, in order to estimate $I(0,\tau,s)$  for $\tan^2\tau\ll X^{-1+\epsilon}$ it is sufficient to consider \eqref{I0 integral def2} with the hypergeometric function being replaced by one. Hence we need to estimate
\begin{equation}\label{I0 n=0 estimate}
\frac{1}{2^4(\cos\tau)^{2-2s}}
\int_{-\infty}^{\infty}r^2h(r)\cosh(\pi r)
\Gamma(1-s+ir)\Gamma(1-s-ir)dr.
\end{equation}
Using the Stirling formula \eqref{Stirling2} and arguing in the same way as in \eqref{Taylor1}-\eqref{Taylor3}, we have
\begin{equation}\label{Taylor5}
\cosh(\pi r)\Gamma(1-s+ir)\Gamma(1-s-ir)=\pi K^{-2it}(1+\ldots).
\end{equation}
Substituting \eqref{Taylor5} to \eqref{I0 n=0 estimate}, we obtain the main term
\begin{equation}\label{I0 n=0 estimate2}
\frac{\pi K^{-2it}}{2^4(\cos\tau)^{2-2s}}
\int_{-\infty}^{\infty}r^2h(r)dr.
\end{equation}
Using \eqref{hN def}, making the change of variable $r=K+Gu$  and applying \eqref{exp integral1}-\eqref{exp integral2}
we show that the main term \eqref{I0 n=0 estimate2} is negligible. Consequently, we prove that
\begin{equation}\label{estimate of n=0 contribution2}
\int_0^{\pi/2}
\left|\frac{1}{G}\int_T^{2T}I(0,\tau,s)dK\right|d\tau\ll
\frac{T^{3/2}X^{\epsilon}}{X^{1/2}}.
\end{equation}
Let us consider the case $n=\pm2.$ It follows from \eqref{x def} and \eqref{I integral hypergeom1} that
\begin{multline}\label{I n=2}
\sum_{n=\pm2}I(n,\tau,s)=
\frac{1}{2}\int_{-\infty}^{\infty}r^2h(r)\cosh(\pi r)\sum_{\pm}
\frac{\Gamma(1-s-ir)\Gamma(1-s+ir)}{4(\cos\tau)^{2-2s}}\\ \times
F\left(1-s+ir,1-s-ir,1;-\frac{(1\pm\sin\tau)^2}{\cos^2\tau}\right)dr.
\end{multline}
Both cases ("+" and "-") of \eqref{I n=2} can be treated in the same way as the case $n=0$. In fact, the only difference between \eqref{I0 integral def2} and \eqref{I n=2} is that $\tan \tau$ in the argument of the hypergeometric function is replaced by $(1\pm\sin\tau)^2/\cos^2\tau.$
The case of $n=\pm1$ can also be treated similarly.

\end{proof}

\bigskip

\noindent\textit{Acknowledgments.}
We thank Gergely Harcos and Han Wu for bringing the paper \cite{Wu} to our attention, and Giacomo Cherubini for pointing out the correct value of the subconvexity exponent.



\nocite{}


\begin{thebibliography}{}

\bibitem{ChatCherLaak}
\newblock
O. Balkanova, D. Chatzakos, G. Cherubini, D. Frolenkov and N. Laaksonen.  \emph{Prime geodesic theorem in the 3-dimensional hyperbolic space}, Trans. Amer. Math. Soc., to appear, arXiv:1712.00880 [math.NT].

\bibitem{BF1}
O. Balkanova and D. Frolenkov. \emph{The mean value of symmetric square
$L$-functions}, Algebra Number Theory, 12 (2018), no. 1, 35--59.

\bibitem{BF2}
O. Balkanova and D. Frolenkov. \emph{Bounds for a spectral exponential sum}, J. London Math. Soc., to appear, arXiv:1803.04201 [math.NT].

\bibitem{BF3}
O. Balkanova and D. Frolenkov. \emph{Convolution formula for the sums of generalized Dirichlet L-functions}, Rev. Mat. Iberoam., to appear, arXiv:1709.01365 [math.NT].

\bibitem{BE}
\newblock
H. Beitman and A. Erdelyi. \emph{Higher transcendental functions}, Vol. 1, McGraw-Hill, New York, 1953.



\bibitem{BrugMot}
\newblock
R. W. Bruggeman and Y. Motohashi. \emph{Sum formula for Kloosterman sums and fourth moment of the Dedekind zeta-function over the Gaussian number field},  Funct. Approx. Com. Math. 31 (2003),  23--92.

\bibitem{CI}
J. B. Conrey and H. Iwaniec. \emph{The cubic moment of central values of automorphic L-functions}, Ann. of Math. (2) 151 (2000), 1175--1216.

\bibitem{EGM}
J. Elstrodt, F. Grunewald, and J. L. Mennicke. \emph{Groups acting on hyperbolic space.} Harmonic analysis and number theory, Springer Monographs in Mathematics. Berlin: Springer, 1998.

\bibitem{GR}
\newblock
I. S. Gradshteyn and I. M. Ryzhik. \emph{ Table of Integrals, Series, and Products}. Edited by A. Jeffrey and D. Zwillinger. Academic Press, New York, 7th edition (2007).

\bibitem{Hough}
\newblock
B. Hough. \emph{Zero-density estimate for modular form L-functions in weight aspect},  Acta. Arith. 154 (2012), no. 2, 187--216.

\bibitem{Ivic}
\newblock
A. Ivic. \emph{On exponential sums with Hecke series at central points},  Funct. Approx. Com. Math. 37 (2007), no. 2, 233--261.


\bibitem{IvJut}
\newblock
A. Ivic and M. Jutila. \emph{On the moments of Hecke series at central points II},  Funct. Approx. Com. Math. 31 (2003),  93--108.

\bibitem{IwPG}
\newblock
 H. Iwaniec. \emph{Prime geodesic theorem}, J. Reine. Angew. Math. 349 (1984), 136--159.
 
\bibitem{Kauf}
\newblock
R. M. Kaufman. \emph{An estimate of Hecke's L-functions of the Gaussian field on the line $\Re(s)=1/2$},  Dokl. Akad. Nauk. BSSR 22 (1978),  25--28. In Russian.


\bibitem{Koyama}
\newblock
S. Y. Koyama. \emph{Prime Geodesic Theorem for the Picard manifold under the mean-Lindel\"{o}f hypothesis},  Forum. Math. 13 (2001), no. 6, 781--793.

\bibitem{MV}
P. Michel and A. Venkatesh.  \emph{The subconvexity problem for GL2}, Publications math\'{e}matiques
de l'IH\'{E}S 111, 1 (June 2010), 171--271.

\bibitem{Mot1992}
\newblock
Y. Motohashi. \emph{Spectral mean values of Maass wave forms}, J. Number Theory 42 (1992), 258--284.

\bibitem{Mot1997}
\newblock
Y. Motohashi. \emph{Trace formula over the hyperbolic upper half-space}, in: Y. Motohashi (Ed.), Analytic Number Theory, Cambridge Univ. Press, Cambridge 1997, 265--286.

\bibitem{Mot2001}
\newblock
Y. Motohashi. \emph{New analytic problems over imaginary quadratic number fields}, in: M. Jutila and T. Mets\"{a}nkyl\"{a} (Eds.), Number Theory, de Gruyter, Berlin, 2001, 255--279.

\bibitem{Nak}
\newblock
M. Nakasuji. \emph{Prime geodesic theorem via the explicit formula for $\Psi$ for hyperbolic 3-manifolds},  Proc. Japan Acad. Ser. A Math. Sci. 77 (2001), no. 7, 130--133.

\bibitem{Nak2}
\newblock
M. Nakasuji. \emph{Prime geodesic theorem via the explicit formula for $\Psi$ for hyperbolic 3-manifolds},  Research report  KSTS/RR-00/005, Keio University 2000,
available at http://www.math.keio.ac.jp/academic/research\_pdf/report/2000/00005.pdf.

\bibitem{Nak3}
\newblock
M. Nakasuji. \emph{Generalized Ramanujan conjecture over general imaginary quadratic
fields}, Forum Math. 24 (2012), no. 1, 85--98.

\bibitem{HMF}
F. W. J.~Olver , D. W.~Lozier, R. F.~Boisvert and C. W.~Clarke. \emph{{NIST} {H}andbook of {M}athematical {F}unctions}, Cambridge University Press, Cambridge (2010).

\bibitem{PR} Y. N. Petridis and M. S. Risager. \emph{Local average in hyperbolic lattice point counting}, with an Appendix by Niko Laaksonen, Math. Z. 285 (2016), no. 3--4, 1319--1344.

\bibitem{PBM}
A. R.~Prudnikov , Yu. A.~Brychkov, O. I.~Marichev. \emph{Integrals and Series, vol.2: Special Functions}, Gordon and Breah Science Publ., Amsterdam (1986).


\bibitem{Sarnak} P. Sarnak. \emph{The arithmetic and geometry of some hyperbolic three manifolds}, Acta Math. 151 (1983), 253--295.
\bibitem{Szmidt}
\newblock
J. Szmidt. \emph{The Selberg trace formula for the Picard group $SL(2\,Z[i])$},  Acta. Arith. 42 (1983), no. 4, 391--424.



\bibitem{Wu}
H. Wu. \emph{Burgess-like subconvexity for $GL_1$}, arXiv:1604.08551 [math.NT].

\bibitem{Z}
\newblock
D. Zagier. \emph{ Modular forms whose Fourier coefficients involve zeta-functions of quadratic fields}. Modular
functions of one variable, VI (Proc. Second Internat. Conf., Univ. Bonn, Bonn, 1976), pp. 105--169.
Lecture Notes in Math., Vol. 627, Springer, Berlin, 1977.





\end{thebibliography}
\end{document}